\documentclass[twoside, final]{amsart}

\usepackage{amsmath, amsfonts, amsthm, amssymb}
\usepackage{exscale}
\usepackage{amscd}

\usepackage[notref, notcite]{showkeys}

\renewcommand{\geq}{\geqslant}
\renewcommand{\leq}{\leqslant}

\numberwithin{equation}{section}
\numberwithin{equation}{subsection}

\newtheorem{proposition}{Proposition}[subsection]

\newtheorem{corollary}[proposition]{Corollary}
\newtheorem{lemma}[proposition]{Lemma}
\newtheorem*{main-theorem}{Main Theorem}
\newtheorem*{theorem*}{Theorem}

\theoremstyle{definition}

\newtheorem{remark}[proposition]{Remark}
\newtheorem{definition}[proposition]{Definition}
\newtheorem*{remark*}{Remark}

\newcommand{\Cal}{\mathcal}

\def\wX {\widetilde{X}}
\def\wY {\widetilde{Y}}
\def\Span {\operatorname{Span}}
\def\Spec {\operatorname{Spec}}
\def\EsSp {\operatorname{Ess\, Spec}}
\def\frX {\mathfrak{X}}
\def\ubar{\overline{u}}
\def\psibar{\overline{\psi}}
\def\phi{\varphi}
\def\tJ {\widetilde{J}}
\def\pa {\partial}
\def\ep {\epsilon}
\def\ZZ{{\mathbb Z}}
\def\NN{{\mathbb N}}
\def\Rr{{\mathbb R}}
\def\reals{{\mathbb R}}

\def\Ci{{\mathcal C}^\infty}
\def\Re{\,\mathrm{Re}\,}

\def\phi{\varphi}

\def\dist{\operatorname{dist}}

\def\be{\begin{eqnarray*}}
\def\ee{\end{eqnarray*}}
\def\ben{\begin{eqnarray}}
\def\een{\end{eqnarray}}

\def\L2R{L_{\text{Rest}}^2}

\def\11{\mathds{1}}

\def\RR{\mathcal{R}}

\def\L2c{L^2_{\text{comp}}}

\def\tR{\tilde{R}}

\def\11{\mathbb{1}}

\def\tE{\widetilde{E}}
\def\Hloc{H_{\text{loc}}}
\def\Lloc{L_{\text{loc}}}
\def\tOmega{\widetilde{\Omega}}

\begin{document}

\title[Bound States]{Nonlinear bound states on weakly
  homogeneous spaces}

\author[H. Christianson]{Hans Christianson}
\email{hans@math.unc.edu}
\address{Department of Mathematics, UNC-Chapel Hill \\ CB\#3250
  Phillips Hall \\ Chapel Hill, NC 27599}

\author[J. Marzuola]{Jeremy Marzuola}
\email{marzuola@math.unc.edu}
\address{Department of Mathematics, UNC-Chapel Hill \\ CB\#3250
  Phillips Hall \\ Chapel Hill, NC 27599}

\author[J. Metcalfe]{Jason Metcalfe}
\email{metcalfe@email.unc.edu}
\address{Department of Mathematics, UNC-Chapel Hill \\ CB\#3250
  Phillips Hall \\ Chapel Hill, NC 27599}

\author[M. Taylor]{Michael Taylor}
\email{met@math.unc.edu}
\address{Department of Mathematics, UNC-Chapel Hill \\ CB\#3250
  Phillips Hall \\ Chapel Hill, NC 27599}

\subjclass[2000]{}

\begin{abstract}

We prove the existence of ground state solutions for a class of
nonlinear elliptic equations, arising in the production of
standing wave solutions to an associated family of nonlinear Schr{\"o}dinger
equations.  We examine two constrained minimization problems, which give rise
to such solutions.  One yields what we call $F_\lambda$-minimizers, the other
energy minimizers.  We produce such ground state solutions on a class of
Riemannian manifolds called weakly homogeneous spaces, and establish
smoothness, positivity, and decay properties.  We also identify classes of
Riemannian manifolds with no such minimizers, and classes for which
essential uniqueness of positive solutions to the associated elliptic
PDE fails.

\end{abstract}

\maketitle

\section{Introduction}\label{s1}

Let $M$ be a Riemannian manifold of dimension $n \geq 2$
(possibly with boundary) with $\Ci$
bounded geometry.  Eventually we will impose the additional assumption
of weak homogeneity, but for now we work in the generality of bounded
geometry.  We consider the nonlinear Schr\"odinger equation on $M$, and
in particular, we are
interested in studying the existence of nonlinear bound states
(standing waves).

The nonlinear Schr{\"o}dinger equations we consider are given by
\begin{equation}
\label{E:NLS}
\begin{cases}
iv_t + \Delta v + |v|^{p-1} v = 0, \ x \in M  \\
v(0,x) = v_0 (x),
\end{cases}
\end{equation}
where $\Delta$ is the Laplace-Beltrami operator on $M$.
If $\pa M\neq\emptyset$, we might impose the Dirichlet or Neumann
boundary condition at $\pa M$.
A nonlinear bound state is a choice of initial condition $u_\lambda(x)$
such that
\ben
v(t,x) = e^{i \lambda t} u_\lambda(x)
\een
satisfies \eqref{E:NLS} with initial data $v(0,x) = u_\lambda(x)$.
Such a solution is also called a ground state, a standing wave, or a solitary
wave, or, sometimes, a soliton.
Plugging in this ansatz yields the following elliptic equation for
$u_\lambda$:
\begin{equation}
\label{E:SNLS}
-\Delta u_\lambda + \lambda u_\lambda - | u_\lambda|^{p-1} u_\lambda = 0.
\end{equation}
Similarly, seeking a standing wave solution to a nonlinear Klein-Gordon
equation,
\ben
v_{tt}-\Delta v+\sigma^2v-|v|^{p-1}v=0,\quad v(t,x)=e^{i\mu t}u(x),
\een
leads to (\ref{E:SNLS}), with $\lambda=\sigma^2-\mu^2$.

In studying (\ref{E:SNLS}), we will assume
\ben
\lambda  > - \delta_0,
\label{E:lam-del}
\een
given that the spectrum of $-\Delta$ is
contained in a semi-infinite interval
\begin{equation}
\label{E:delta-spec}
\Spec (-\Delta ) \subset [\delta_0, \infty ),
\end{equation}
for some $\delta_0 \geq 0$.

We will analyze two methods of establishing the existence of a solution to
(\ref{E:SNLS}).  One is to mimimize the functional
\begin{equation}
\label{E:F-lambda}
F_\lambda (u) = \| \nabla u \|^2_{L^2} + \lambda \| u \|^2_{L^2}
\end{equation}
subject to the constraint
\begin{equation}
\label{E:J-p}
J_p (u) = \int_M | u |^{p+1} dV = \beta,
\end{equation}
with $\beta\in (0,\infty)$ fixed.
For this, we will require
\begin{equation}
\label{E:p-range}
p\in \Bigl(1,1+\frac{4}{n-2}\Bigr),\quad \text{i.e., }\
p+1 \in \left( 2, \frac{2n}{n-2} \right).
\end{equation}

The other is to minimize the energy
\ben
E(u)=\frac{1}{2}\|\nabla u\|^2_{L^2}-\frac{1}{p+1}\int_M |u|^{p+1}\, dV,
\label{1.0.10}
\een
subject to the constraint on the ``mass''
\ben
Q(u)=\|u\|_{L^2}^2=\beta,
\label{E:Q-beta}
\een
with $\beta\in (0,\infty)$ fixed.  For this, we require the more stringent
condition
\ben
p\in \Bigl(1,1+\frac{4}{n}\Bigr).
\label{E:pp}
\een
The energy functional (\ref{1.0.10}) is conserved for sufficiently
regular solutions to the nonlinear Schr{\"o}dinger equation (\ref{E:NLS}),
which imparts special importance to energy minimizers.

We preview these approaches in more detail.

\subsection{$F_\lambda$ minimizers}

We make the hypotheses (\ref{E:lam-del}) and (\ref{E:p-range}),
and desire to minimize $F_\lambda(u)$ over $u\in H^1(M)$,
subject to the constraint (\ref{E:J-p}).
We take $H^1(M)$ to be the $L^2$-Sobolev space of complex valued functions
on $M$, however, with the structure of a vector space over $\Rr$.
Here and below, if $\pa M\neq\emptyset$ and we impose the Dirichlet
condition, we take $H^1(M)$ to mean $H^1_0(M)$.  In certain cases, such as
\S{\ref{s2.a}}, where the Dirichlet boundary condition is central, we
use the notation $H^1_0(M)$, for emphasis.

If $u, v \in H^1(M)$, we have
\ben
\aligned
\frac{d}{d\tau} F_\lambda ( u + \tau v)\Bigr|_{\tau=0}
&= 2 \Re ( -\Delta u + \lambda u , v ), \\
\frac{d}{d\tau} J_p(u+ \tau v) \Bigr|_{\tau=0}
&=(p+1)\Re \int | u |^{p-1} u \overline{v}\, dV.
\endaligned
\label{1.1.1}
\een
If $u \in H^1(M)$ is a constrained minimizer, then
\ben
\aligned
&v \in H^1(M) \text{ and } \Re \int_M | u|^{p-1} u \overline{v}\, dV = 0 \\
&\implies \Re (-\Delta u + \lambda u, v ) = 0,
\endaligned
\een
so the two elements of $H^{-1}(M)$, $-\Delta u + \lambda u$ and
$| u |^{p-1} u$, are linearly dependent over $\Rr$.
Hence there exists a real constant $K_0$ such that
\begin{equation}
\label{E:stat-K0}
-\Delta u + \lambda u = K_0 | u |^{p-1} u,
\end{equation}
with equality holding in $H^{-1}(M)$.  To determine $K_0$,
we pair each side of this equation with $u$ and use $H^1 - H^{-1}$
duality:
\ben
\| \nabla u \|^2_{L^2} + \lambda \| u \|_{L^2}^2 = K_0 \int_M | u
|^{p+1}\, dV = \beta K_0.
\label{1.1.4}
\een
Hence
\begin{equation}
\label{E:K0}   
K_0 = \frac{1}{\beta} \inf\, \{ F_\lambda (u) : u \in H^1(M), J_p(u) =
\beta \}.
\end{equation}
Given the existence of such an infimizer, it follows that
\ben
K_0>0.
\een
Now if $u$ solves \eqref{E:stat-K0}, then $u_a(x) = a u(x)$ solves
\ben
-\Delta u_a + \lambda u_a = | a |^{-(p-1)} K_0 | u_a |^{p-1} u_a,
\een
so that we can solve
\ben
-\Delta u + \lambda u = K | u |^{p-1} u
\een
for any $K>0$.

\subsection{Energy minimizers}

We make the hypothesis (\ref{E:pp}) on $p$ and desire to minimize
$E(u)$, subject to the constraint (\ref{E:Q-beta}).  If $u,v\in H^1(M)$,
we have
\ben
\aligned
\frac{d}{d\tau}E(u+\tau v)\Bigr|_{\tau=0}
&=\Re (-\Delta u-|u|^{p-1}u,v), \\
\frac{d}{d\tau}Q(u+\tau v)\Bigr|_{\tau=0}
&=2\, \Re (u,v).
\endaligned
\label{1.2.1}
\een
If $u\in H^1(M)$ is a constrained minimizer, then
\ben
\aligned
v\in\ &H^1(M),\quad \Re (u,v)=0 \\
&\Longrightarrow \Re (\Delta u+|u|^{p-1}u,v)=0,
\endaligned
\label{1.2.2}
\een
and it follows that there exists $\lambda\in\Rr$ such that $\Delta u+
|u|^{p-1}u=\lambda u$, or equivalently,
\ben
-\Delta u+\lambda u=|u|^{p-1}u.
\label{1.2.3}
\een

\subsection{Background}

Before describing the structure of the main body of this paper, we recall
some previous work on ground state solutions to (\ref{E:SNLS}), and describe
how we plan to extend the scope.  There is a large literature on such problems,
when $M$ is Euclidean space $\Rr^n$ or a bounded domain in $\Rr^n$, addressing
questions of existence, uniqueness, and stability.  We mention a few of these,
referring to them for further references.

Pohozaev \cite{Poh} studied the
case where $M$ is a bounded domain in $\Rr^n$, for $p$ as in (\ref{E:p-range})
(also allowing for more general nonlinearities), producing
$F_\lambda$-minimizers.  These results extend readily to general compact
Riemannian manifolds with boundary.  Strauss \cite{Str} obtained solutions
on $\Rr^n$, essentially via $F_\lambda$-minimizers.
This work was followed by \cite{BPL}.  This paper
constructs a minimizer of
$\|\nabla u\|^2_{L^2}$, subject to the constraint
\ben
\int_{\Rr^n} G(u(x))\, dx=1,
\label{1.0.12}
\een
where
\ben
G(u)=\int_0^u g(t)\, dt,\quad g(u)=|u|^{p-1}u-\lambda u.
\een
Both \cite{Str} and \cite{BPL} considered more general nonlinearities.
A key device in both papers was a symmetrization technique:
if $u\in H^1(\Rr^n)$,
then replacing $u$ by its decreasing radial rearrangement $u^*$ leaves
$\|u\|_{L^2}$ and $\|u\|_{L^{p+1}}$ unchanged and does not increase
$\|\nabla u\|_{L^2}$.
In \cite{Str} this was demonstrated directly.  As noted in \cite{BPL},
it also follows from \cite{BLL}, via an argument using heat kernel 
monotonicity and a rearrangement inequality.
From this result, it suffices to seek a minimizer within the class
of {\it radial} functions in $H^1(\Rr^n)$, where estimates implying compactness
are available.
The papers \cite{L1}--\cite{L2} introduced the concentration-compactness
method and applied it to a number of problems, including a construction
of $F_\lambda$-minimizers, and also a discussion of energy minimizers.

Another approach was taken in \cite{Wei83}.  There a solution to
(\ref{E:SNLS}) was constructed to maximize the Weinstein functional
\ben
W(u)=\frac{\|u\|^{p+1}_{L^{p+1}}}
{\|u\|_{L^2}^\alpha \|\nabla u\|_{L^2}^\beta},
\label{1.3A.3}
\een
over $u\in H^1(\Rr^n)$, under the hypothesis (\ref{E:p-range}),
where
\ben
\alpha=2-\frac{(n-2)(p-1)}{2},\ \
\beta=\frac{n(p-1)}{2},
\label{1.3A.4}
\een
(Note that $\alpha,\beta>0$ and $\alpha+\beta=p+1$.)
The supremum $W_{\max}$ is the best constant in the Gagliardo-Nirenberg
inequality
\ben
\|u\|_{L^{p+1}}^{p+1}\le W_{\max} \|u\|^\alpha_{L^2}
\|\nabla u\|_{L^2}^\beta,\quad u\in H^1(\Rr^n).
\label{1.3A.5}
\een
See also Appendix B of \cite{Tao-book} for a presentation of this work,
and \S{\ref{sa4x}} of this paper for another demonstration of
the existence of a maximizer.

There are infinitely many solutions to (\ref{E:SNLS})
on $\Rr^n$, but the ground state is (up to a constant factor) positive
everywhere, and there are results on uniqueness (up to isometries) of positive
solutions, culminating in \cite{Kwo}, \cite{Mc}, and \cite{CLJ}.  Work on
orbital stability of solutions to the nonlinear Schr{\"o}dinger equation
(\ref{E:NLS}) with initial data given by such ground states includes \cite{CL},
\cite{Wei85}, \cite{Wei86}, and \cite{GrShSt}.  The hypothesis that the
ground state be energy minimizing plays a major role in these results, and
this motivates our interest in establishing existence of energy minimizers
as well as $F_\lambda$-minimizers.  (Stability issues for $F_\lambda$-minimizers
that are not energy minimizers are different; cf.~\cite{S1}.)

Moving beyond the cases of bounded domains and Euclidean space, the case where
$M$ is hyperbolic space $\Cal{H}^n$, a symmetric space of constant negative
sectional curvature, was treated independently in \cite{ManSan} and by two of
us, in \cite{ChMa-hyp}.  The techniques in these two papers are rather
different.  Following \cite{BPL}, the paper \cite{ChMa-hyp} finds a minimizer
of $\|\nabla u\|^2_{L^2}$, subject to the constraint (\ref{1.0.12}).  It
uses a mixture of a symmetrization technique and a concentration-compactness
argument.  In this case, the fact that $u\mapsto u^*$ does not increase
$\|\nabla u\|_{L^2}$ was established in \cite{ChMa-hyp}, via
heat kernel monotonicity in the setting of
$\Cal{H}^n$ plus the extension of the rearrangement inequality of
\cite{BLL}, to the setting of hyperbolic space, given in \cite{Bek}
(following the extension to the sphere given in \cite{BeTa}).
The paper \cite{ManSan} tackles existence via a minimization process
essentially equivalent to finding an $F_\lambda$-minimizer, and makes use of
the Ekeland principle.  Their paper also has existence results for the case
when $-\lambda$ is at the bottom of the spectrum of $-\Delta$, and also
results on both existence and non-existence for critical $p=(n+2)/(n-2)$.
Furthermore, \cite{ManSan} establishes uniqueness of positive solutions (up to
isometries) in this setting.

In this paper, we work on the following class
of complete Riemannian manifolds $M$, possibly with boundary.
We assume there is a group $G$ of isometries
of $M$ and a number $D>0$ such that for every $x,y\in M$, there exists
$g\in G$ such that $\text{dist}(x,g(y))\le D$.  We call such a manifold $M$
a {\it weakly homogeneous space.}
We give some examples of such spaces.

First, if $M$ is compact (perhaps with boundary), such a condition holds,
even if only the identity map on $M$ is an isometry.  Next, if $M$
is a homogeneous space, that is, $M$ has a transitive group of isometries
(such as $\Rr^n$ or $\Cal{H}^n$, among many other richly studied examples),
such a condition holds.  We mention some examples that are neither
compact nor homogeneous.

$\text{}$ \newline
{\sc Example 1.} Pick $R\in (0,1/2)$ and take
$$
M=\Rr^n\setminus\bigcup\limits_{k\in\ZZ^n} B_R(k).
$$

\noindent
{\sc Example 2.} Let $M^b$ be a compact Riemannian manifold (perhaps with
boundary), and let $M$ be the universal covering space of $M^b$, with the
induced metric tensor.  More generally, $M$ can be any covering space of
$M^b$.  This class of examples includes Example 1 as a special case.
\newline $\text{}$

\subsection{Plan of the paper}

In \S{\ref{s2}}, we prove existence of a minimizer $u\in H^1(M)$ to
$F_\lambda(u)$, subject to the constraint (\ref{E:J-p}), given $p$ as in
(\ref{E:p-range}), and establish some useful properties of such solutions,
when $M$ is a weakly homogeneous space.  The first use of the constraint
on $p$ is to get
\ben
I_\beta=\inf \{F_\lambda(u):u\in H^1(M),\,J_p(u)=\beta\}>0.
\label{1.3.1}
\een
We then apply the concentration-compactness method of P.-L.~Lions,
suitably extended to the manifold setting.  Concentration is established
in \S{\ref{s2.1}}, and from there compactness and existence of
$F_\lambda$-minimizers is proven in \S{\ref{s2.2}}.
The concentration argument works whenever $M$ has bounded geometry.
It is in passing to the compactness argument that we use the weak
homogeneity.  Regarding the necessity of some geometrical constraint,
we give examples in \S{\ref{s2.a}}
of Riemannian manifolds (with boundary) with bounded
geometry, for which there is not such an $F_\lambda$-minimizer.

Section \ref{s2.3} establishes smoothness of minimizers, and \S{\ref{s2.4}}
is devoted to showing that every real valued minimizer $u$ is either $>0$
on $M$ or $<0$ on $M$, and obtaining some Harnack-type estimates.
Further global bounds on positive minimizers are obtained in
\S{\ref{s2.5}}.  In \S\S{\ref{s2.3}--\ref{s2.5}} we assume for simplicity
that $\pa M=\emptyset$.  Some local regularity estimates for nonempty boundary
are given in \cite{Tay11}, in the setting of Lipschitz domains.

In \S{\ref{s3}} we prove existence of a minimizer $u\in H^1(M)$ to the energy
$E(u)$, subject to the constraint (\ref{E:Q-beta}), given $p$ as in
(\ref{E:pp}), when $M$ is a weakly homogeneous space, under one further
condition.  Namely, we require
\ben
\Cal{I}_\beta=\inf \{E(u):u\in H^1(M), Q(u)=\beta\}
\label{1.3.2}
\een
to satisfy
\ben
\Cal{I}_\beta<0.
\label{1.3.3}
\een
(Replace $H^1(M)$ by $H^1_0(M)$ when using the Dirichlet boundary condition.)
Again we use a concentration-compactness argument.  In \S{\ref{s3.1}}
we show that (\ref{1.3.3}) forces concentration, and we discuss existence
of energy minimizers for this in \S{\ref{s3.2}}.  We mention that energy
minimizers are (constant multiples of) positive functions on $M$, and
many of the estimates of \S\S{\ref{s2.3}--\ref{s2.5}} apply here,
as seen in \S{\ref{s3.4}}, following a discussion of manifolds with no
energy minimizers in \S{\ref{s3.3}}.  We give a formula for the second
variation of energy, for energy minimizers, in \S{\ref{s3.5}}
and for $F_\lambda$-minimizers in \S{\ref{s3.a}}
and apply these formulas in \S{\ref{s3.6}} to results on operators
$L_{\pm}$.  These results in turn are applied in \S{\ref{s3.8}} to
some results on orbital stability.

In \S{\ref{s4}} we take another look at the symmetrization method,
mentioned above in the context of $\Rr^n$.
After a discussion of the rearrangement lemma in \S{\ref{s4.1}},
in \S{\ref{sa4}} we pursue a uniform approach to a proof of existence of
$F_\lambda$-minimizers, valid for $M=\Cal{H}^n$ and for $\Rr^n$.  This proof
is adapted from \cite{ChMa-hyp}, but it
incorporates simplifications that allow us to avoid completely the
concentration-compactness argument in this context.
In \S{\ref{sa4x}} we apply the symmetrization approach to maximizing $W(u)$
in (\ref{1.3A.3}).
In \S{\ref{sa5}} we discuss such a symmetrization approach to the existence
of energy minimizers, when $M=\Cal{H}^n$ or $\Rr^n$ ($n\ge 2$).
We see that this approach succeeds up to a point, but there appears a gap,
and it remains to be seen whether the concentration-compactness argument
can be avoided in this setting.

We have four appendices, containing supporting material related to the
results of \S\S{\ref{s2}--\ref{s4}}.  As advertised above, Appendix
\ref{sa1} establishes a version of the Lions concentration-compactness
argument.  We work in a general class of measured metric spaces, which in
particular includes the Riemannian manifolds with bounded geometry arising
in \S\S{\ref{s2}--\ref{s3}}.

In \S{\ref{sa3}} we look at the behavior of the energy of positive solutions
to (\ref{E:SNLS}) on Euclidean space $\Rr^n$.
We show that if $1<p<1+4/n$ then all such solutions
(with $\lambda>0$) have negative energy and
are energy minimizing within their mass class,
while if $1+4/n<p<(n+2)/(n-2)$, there can be no energy minimizer within
its mass class, at any positive mass.
In \S{\ref{sa2}}, we discuss conditions
under which $F_\lambda$-minimizers can be shown to have positive energy
($E(u)>0$), even when $1<p<1+4/n$.
In particular, we show that all such minimizers on hyperbolic
space $\Cal{H}^n$ associated to $\lambda\le 0$ have positive energy.

In \S{\ref{sa6}} we examine some positive solutions to (\ref{E:SNLS})
that are not $F_\lambda$-minimizers, and exhibit some cases of essential
non-uniqueness of such positive solutions.



\section{$F_\lambda$ minimizers}\label{s2}

We take up the task of minimizing $F_\lambda(u)$, subject to
the constraint $J_p(u)=\beta$.
Observe that \eqref{E:lam-del}--\eqref{E:delta-spec} imply
\ben
F_\lambda (u) \simeq \| u \|^2_{H^1(M)}.
\een
The hypothesis (\ref{E:p-range}) on the range of $p$ implies
\ben
H^1 (M) \subset L^{p+1}(M).
\een
Hence there exists a constant $C>0$ such that
\ben
\| u \|_{L^{p+1}}^2 \leq C F_\lambda(u),
\een
so that
\begin{equation}
\label{E:I-beta}
I_\beta =  \inf \{ F_\lambda (u) :u\in H^1(M),\, J_p(u) = \beta \} >0.
\end{equation}
(Recall from \S{\ref{s1}} that $H^1(M)$ stands for $H^1_0(M)$
if $\pa M\neq\emptyset$ and we use the
Dirichlet boundary condition.)

Suppose $\{ u_\nu \} \subset H^1(M)$ is a sequence satisfying
\ben
J_p(u_\nu) = \beta,\quad F_\lambda (u_\nu ) \leq I_\beta + \frac{1}{\nu}.
\label{2.0.5}
\een
Passing to a subsequence if necessary, we have
\ben
u_\nu \to u \in H^1(M)
\een
converging in the weak topology.
Rellich's theorem implies
\begin{equation}
\label{E:Rellich}
H^1 (M) \hookrightarrow L^{p+1}(\Omega)
\end{equation}
is compact provided $\Omega$ is relatively compact.
This implies that, for such $\Omega$,
\ben
u_\nu \to u \text{ in the } L^{p+1}(\Omega) \text{ norm.}
\een
If our manifold $M$ is compact, we can take $M = \Omega$ and the
minimization problem is simple.  We are interested in the
non-compact setting, for which further argument is necessary.

To carry this out, we use the
concentration-compactness method of Lions \cite{L1,L2}.
We set up the concentration-compactness argument in \S{\ref{s2.1}} and
establish concentration.  In \S{\ref{s2.2}} we establish compactness
and prove existence of $F_\lambda$-minimizers,
when $M$ is a weakly homogeneous space.
In \S{\ref{s2.a}} we exhibit some manifolds with bounded geometry on which
there are no $F_\lambda$-minimizers.
In \S\S{\ref{s2.3}--\ref{s2.5}} we study smoothness, positivity, and
decay estimates on the solutions; in these last three sections we assume
for simplicity that $\pa M=\emptyset$.

\subsection{Concentration}\label{s2.1}
\label{S:conc}

In this section we again assume only that $M$ is a smooth manifold with
$\Cal{C}^\infty$ bounded geometry.  The enemy of finding a minimizer is the
possibility of minimizing sequences escaping to spatial infinity.  We
will show that some minimizing subsequence concentrates, which we will
define shortly.  

Let us first record a basic Lemma of concentration compactness.
This is given in  Lions \cite{L1}, pp.~115--117,
in the context of Euclidean space.
Here we state the result in the context of a Riemannian manifold.
In Appendix \ref{sa1}, we establish the result in an even more general
setting.

\begin{lemma}
Let $M$ be a Riemannian manifold.  Fix $\beta\in (0,\infty)$.
Let $\{u_\nu \} \in L^{p+1}(M)$ be a sequence satisfying $\int | u_\nu
|^{p+1} dV = \beta$.  Then, after extracting a subsequence, one of the
following three cases holds:

(i) {\it Vanishing.}  If $B_R(y) = \{ x  \in M : \dist (x,y) \leq R
\}$ is the closed ball, then for all $0 < R < \infty$, 
\begin{equation}
\label{E:vanishing}
\lim_{\nu \to \infty} \sup_{y \in M} \int_{B_R(y)} | u_\nu|^{p+1}\, dV =
0.
\end{equation}

(ii) {\it Concentration.}  There exists a sequence of points $\{y_\nu
\} \subset M$ with the property that for each $\epsilon>0$, there
exists $R(\epsilon) < \infty$ such that
\ben
\int_{B_{R(\epsilon)}(y_\nu) } | u_\nu |^{p+1}\, dV > \beta - \epsilon.
\een

(iii) {\it Splitting.}  There exists $\alpha \in (0, \beta)$ with the
following properties:  For each $\epsilon>0$, there exists $\nu_0 \geq
1$ and sets $E_\nu^\sharp, E_\nu^b \subset M$ such that
\begin{equation}
\label{E:splitting-1}
\dist ( E_\nu^\sharp, E_\nu^b ) \to \infty \text{ as } \nu \to \infty,
\end{equation}
and
\begin{equation}
\label{E:splitting-2}
\left| \int_{E_\nu^\sharp} | u_\nu |^{p+1}\, dV - \alpha \right| <
\epsilon, \,\,\, \left| \int_{E_\nu^b} | u_\nu |^{p+1}\, dV - (\beta-\alpha)
\right| <
\epsilon.
\end{equation}

\end{lemma}

We now show that for a minimizing sequence in our problem, the
vanishing phenomenon cannot occur.  The proof uses the following lemma,
essentially given in \cite{L2}.

\begin{lemma}
\label{L:vanishing}
Assume $\{ u_\nu \}$ is bounded in $H^1(M)$, and 
\begin{equation}
\label{E:vanishing-2}
\lim_{\nu \to \infty} \sup_{y \in M} \int_{B_R(y)} | u_\nu|^{2}\, dV =
0, \text{ for some } R>0.
\end{equation}
Then 
\begin{equation}
\label{E:vanishing-3}
2 <  r < \frac{2n}{n-2} \implies \| u_\nu \|_{L^r(M)} \to 0.
\end{equation}
\end{lemma}

\begin{proof}
This is a special case of Lemma I.1 on p. 231 of \cite{L2}.
Actually, the lemma there is established for $M=\Rr^n$.
However, the only two geometric properties
used in the proof in \cite{L2} are the existence of Sobolev embeddings
on balls of radius $R>0$, and the fact that there exists $m < \infty$
such that $\reals^n$ has a covering by balls of radius $R$ in such a
way that each point is contained in at most $m$ balls.  These two
facts hold on every smooth Riemannian manifold with $\Cal{C}^\infty$
bounded geometry.

\end{proof}

\begin{proposition}
\label{P:no-vanishing}
Suppose $\{ u_\nu \}$ is a minimizing sequence.  Then no subsequence
can satisfy \eqref{E:vanishing}.

\end{proposition}

\begin{proof}

If \eqref{E:vanishing} holds, then so does \eqref{E:vanishing-2}, by
H\"older's inequality on finite measure balls.  Then
\eqref{E:vanishing-3} holds with $r = p+1$ (recall (\ref{E:p-range})).
This contradicts the assumption that $J_p(u) = \beta >0$.
\end{proof}

To show that splitting is impossible, we start by showing that $I_\beta$,
given by (\ref{1.3.1}), has the property that, for all $\beta>0$,
\ben
I_\beta < I_\eta + I_{\beta - \eta}, \text{ for any } \eta \in (0,
\beta).
\label{2.1.7}
\een
Lions gives (\ref{2.1.7}) a key role in results of \cite{L1}--\cite{L2},
showing that, in various situations, splitting cannot occur.  In fact, in this
case we have much more structure.

\begin{proposition}
\label{P:subadd}
For all $\beta>0$, we have
\ben
I_\beta = I_1 \beta^{2/(p+1)}.
\een
\end{proposition}

\begin{proof}
Suppose $u_\nu$ satisfies (\ref{2.0.5}), so
\ben
J_p(u_\nu) = \beta,\quad F_\lambda(u_\nu)\rightarrow I_\beta.
\een
Then for $a>0$, 
\ben
J_p(a u_\nu) = a^{p+1} \beta,\quad F_\lambda(au_\nu)\rightarrow
a^2 I_\beta.
\een
Hence
\ben
\gamma=a^{p+1}\beta\Longrightarrow
I_\gamma \leq a^2 I_\beta.
\label{2.1.11}
\een
Now if we replace $\beta$ by $\gamma$ and $a$ by $a^{-1}$, we get
$I_\beta\leq a^{-2}I_\gamma$, i.e., $I_\gamma\geq a^2 I_\beta$,
which together with (\ref{2.1.11}) implies
\ben
\gamma = a^{p+1} \beta \implies I_\gamma = a^2 I_\beta,
\een
and proves the proposition.
\end{proof}

Note that since $C F_\lambda (u) \geq \int | u |^{p+1} d V$, we have
$I_1 >0$, so $I_\beta >0$ for every $\beta>0$, and (\ref{2.1.7}) follows.

$\text{}$ \newline
{\bf Remark.} The proof of Proposition \ref{P:subadd} implies that
$F_\lambda$-minimizers scale to other $F_\lambda$-minimizers.  We state this
formally.

\begin{corollary}\label{c2.1.R}
If $\beta>0$ and $u$ minimizes $F_\lambda$, subject to the constraint
$J_p(u)=\beta$, and if $\kappa>0$, then $u_\kappa=\kappa u$ minimizes
$F_\lambda$, subject to the constraint $J_p(u_\kappa)=\kappa^{p+1}\beta$.
\end{corollary}

Let us now show that splitting cannot occur for a minimizing
subsequence.  Suppose on the contrary that there exists $\alpha \in
(0, \beta)$ and for each $\epsilon>0$, sets $E_\nu^\sharp, E_\nu^b
\subset M$ such that \eqref{E:splitting-1}-\eqref{E:splitting-2}
occur.  Choose $\epsilon>0$ sufficiently small that 
\begin{equation}
\label{E:C-1-eps}
I_\beta < I_\alpha + I_{\beta - \alpha} - C_1 \epsilon,
\end{equation}
where $C_1 >0$ is a sufficiently large constant to be fixed later.
Since $\| u_\nu \|_{H^1(M)}$ is uniformly bounded, if follows from
\eqref{E:splitting-1} that there exists $\nu_1$ such that $\nu \geq
\nu_1$ implies
\ben
\int_{S_\nu} | u_\nu |^2\, dV < \epsilon,
\een
where $S_\nu$ is a set of the form
\ben
S_\nu = \{ x \in M : d_\nu < \dist(x, E_\nu^\sharp) \leq d_\nu + 2 \}
\subset M \setminus (E_\nu^\sharp \cup E_\nu^b),
\een
for some $d_\nu >0$.  For $r>0$ and $\nu \geq \nu_1$, set 
\ben
\tE_\nu(r) = \{ x \in M : \dist(x, E_\nu^\sharp ) \leq r \},
\een
so that $S_\nu=\tE_\nu(d_\nu +2) \setminus \tE (d_\nu).$  Define functions
$\chi_\nu^\sharp$ and $\chi_\nu^b$ by 
\ben
\chi_\nu^\sharp(x) = \begin{cases}
1,\quad \text{ if } x \in \tE_\nu (d_\nu), \\
1 - \dist(x, \tE_\nu ( d_\nu ) ),\quad \text{ if } x \in \tE_\nu ( d_\nu +
1 ), \\
0,\quad \text{ if } x \notin \tE_\nu ( d_\nu + 1 ),
\end{cases}
\een
and
\ben
\chi_\nu^b(x) = \begin{cases}
0,\quad \text{ if } x \in \tE_\nu (d_\nu + 1), \\
\dist(x, \tE_\nu ( d_\nu ) ),\quad \text{ if } x \in \tE_\nu ( d_\nu +
2 ), \\
1,\quad \text{ if } x \notin \tE_\nu ( d_\nu + 2 ),
\end{cases}
\een
These functions are both Lipschitz with Lipschitz constant $1$ and
almost disjoint supports.  Set $u_\nu^\sharp = \chi_\nu^\sharp u_\nu$ and
$u_\nu^b = \chi_\nu^b u_\nu$.  Since $0 \leq \chi_\nu^\sharp +
\chi_\nu^b \leq 1$, we have
\ben
J_p(u_\nu^\sharp) + J_p (u_\nu^b) \leq \int(\chi_\nu^\sharp +
\chi_\nu^b) | u_\nu|^{p+1}\, d V \leq J_p(u_\nu) = \beta.
\een
Also, of course if $\lambda\geq 0$,
\ben
\lambda \| u_\nu^\sharp \|_{L^2}^2 + \lambda \| u_\nu^b
\|_{L^2}^2  \leq \lambda \| u_\nu \|_{L^2}^2 ,
\een
while if $\lambda<0$, we have
\ben
| \lambda | \left( \| u_\nu \|_{L^2}^2 - \| u_\nu^\sharp \|_{L^2}^2 -
  \| u_\nu^b \|_{L^2}^2 \right) < |\lambda | \epsilon.
\een
We have $\nabla u_\nu^\sharp = \chi_\nu^\sharp \nabla u_\nu + (\nabla
\chi_\nu^\sharp) u_\nu$ and similarly for $u_\nu^b$, and $| \nabla
\chi_\nu^\sharp | \leq 1$ except for a set of measure zero, so 
\ben
\| \nabla u_\nu^\sharp \|_{L^2}^2 + \| \nabla u_\nu^b \|_{L^2}^2 \leq
\| \nabla u_\nu \|_{L^2}^2 +  \int_{S_\nu} | u_\nu |^2\, dV.
\een
As a consequence, we have
\begin{equation}
\label{E:F-lambda-sharp}
F_\lambda ( u_\nu^\sharp ) + F_\lambda ( u_\nu^b)  \leq F_\lambda
(u_\nu) +  \epsilon + \sigma(\lambda) \epsilon,
\end{equation}
where 
\ben
\sigma(\lambda) =
\begin{cases} | \lambda |,\quad -\delta_0 < \lambda < 0,  \\
0,\quad \lambda \geq 0.
\end{cases}
\een
Using the support properties of $u_\nu^\sharp$, $u_\nu^b$ together
with \eqref{E:splitting-1}-\eqref{E:splitting-2} yields
\begin{equation}
\label{E:J-p-sharp}
| J_p (u_\nu^\sharp) - \alpha |,\quad | J_p(u_\nu^b) - (\beta -
\alpha) | \leq 3 \epsilon.
\end{equation}
Combining \eqref{E:F-lambda-sharp}-\eqref{E:J-p-sharp}, and letting
$\nu \to \infty$, we get
\ben
I_\alpha + I_{\beta - \alpha} \leq I_\beta + C_2 \epsilon,
\een
where $C_2$ depends only on $\delta_0>0$, the bottom of the spectrum
of $-\Delta$.  Hence if $C_1>C_2$ is chosen sufficiently large in
\eqref{E:C-1-eps} (which simply amounts to producing $\ep$ sufficiently
small), we contradict \eqref{E:C-1-eps}.  This, together
with Proposition \ref{P:no-vanishing} yields the following
proposition, which states that for a minimizing sequence $u_\nu$, only
the concentration phenomenon can occur.

\begin{proposition}
\label{P:only-concentration}
Let $\{ u_\nu \}$ be a minimizing sequence of $I_\lambda$.  Then every
subsequence of the $\{u_\nu \}$ has a further subsequence (which we
will continue to denote by $\{ u_\nu \}$) with the following property.
There exists a sequence $\{ y_\nu \} \subset M$ and a function
$\tR(\epsilon )$ such that for all $\nu$,
\begin{equation}
\label{E:only-concentration}
\int_{B_{\tR(\epsilon)}(y_\nu)} | u_\nu|^{p+1}\, dV > \beta - \epsilon,
\quad \forall\, \epsilon>0,
\end{equation}
\end{proposition}

\begin{remark}

It is very important to observe the following facts in Proposition
\ref{P:only-concentration}.  The sequence of points $\{ y_\nu \}$ is
independent of $\epsilon>0$, and the function $\tR(\epsilon)$ is
independent of the index $\nu$.

Proposition \ref{P:only-concentration} is about {\it concentration}
along subsequences of a minimizing sequence.  It holds on any $\Ci$
manifold $M$ with bounded geometry.  In order to show that a minimizer
actually exists, we need {\it compactness}, which will follow once we
assume the additional structure of {\it weak homogeneity} on $M$.
This will allow us to map the sequence $\{y_\nu \}$ into a compact
region so that any subsequence which concentrates as in Proposition
\ref{P:only-concentration} will enjoy {\it compact} Sobolev embeddings
by Rellich's theorem.
\end{remark}

\subsection{Compactness and existence of minimizers}\label{s2.2}

For this section, let $M$ be a smooth Riemannian manifold of dimension
$n \geq 2$, satisfying
the following weak homogeneity condition:

\begin{definition}
Let $M$ be a smooth, complete
Riemannian manifold, possibly with boundary.  We say
$M$ is {\it weakly homogeneous} if there exists a group $G$ of
isometries of $M$ and a number $D>0$ such that for every $x, y \in M$,
there exists an element $g \in G$ such that $\dist(x, g(y) ) \leq D$.
\end{definition}

\noindent
{\bf Remark.} Such Riemannian manifolds necessarily have $\Cal{C}^\infty$
bounded geometry.
\newline $\text{}$

We retain the hypotheses (\ref{E:delta-spec})--(\ref{E:p-range}).
Let $\{u_\nu \} \subset H^1(M)$ be a minimizing sequence for \eqref{E:I-beta},
that is,
\ben
J_p( u_\nu) = \beta, \,\,\, F_\lambda(u_\nu ) \leq I_\beta +
\frac{1}{\nu},
\een
where $F_\lambda$ is given by \eqref{E:F-lambda} and $J_p$ is given by
\eqref{E:J-p}
as usual.  After passing to a subsequence if necessary, Proposition
\ref{P:only-concentration} shows we have points $\{ y_\nu \} \subset
M$ and a function $\tR(\epsilon)$ such that
\eqref{E:only-concentration} holds.  

We now fix a base point, or ``origin'' $o \in M$ and apply the weak
homogeneity hypothesis: for each $\nu$, there exists $g_\nu \in G$ and
$x_\nu \in B_D(o)$ such that $x_\nu = g_\nu ( y_\nu)$.  Set $v_\nu(x)
= u_\nu(g_\nu^{-1} (x))$, so that $v_\nu$ is now concentrated near $o$
instead of near $y_\nu$.  The sequence $\{ v_\nu \}$ satisfies
\ben
J_p(v_\nu) = J_p(u_\nu) = \beta, \,\,\, F_\lambda(v_\nu ) = F_\lambda
(u_\nu ) \leq I_\beta + \frac{1}{\nu},
\een
and
\begin{equation}
\label{E:conc-int}
\int_{B_{\tR(\epsilon) + D }(o) } |v_\nu|^{p+1}\, dV > \beta - \epsilon,
\,\,\, \forall\, \epsilon>0.
\end{equation}
Passing to a further subsequence, which we continue to denote by $\{
v_\nu \}$, we have
\ben
v_\nu \to v\ \text{ weakly in }\ H^1(M).
\label{2.2.3A}
\een
Since
$F_\lambda$ is comparable to the $H^1(M)$ norm squared, we have 
\ben
F_\lambda(v) \leq \liminf_{\nu \to \infty} F(v_\nu ) = I_\beta.  
\een
Similarly
\ben
J_p(v) \leq \liminf_{\nu \to \infty} J_p(v_\nu) = \beta.
\een
On the other hand, by \eqref{E:Rellich}, we have for each $\epsilon>0$, 
\ben
v_\nu \to v, \text{ in the } L^{p+1}(B_{\tR(\epsilon) + D}(o)) \text{
  norm,}
\een
so that \eqref{E:conc-int} implies $J_p(v) \geq \beta$.  Hence $J_p(v)
= \beta$, which in turn implies
\ben
F_\lambda (v) = I_\beta.
\label{2.2.7}
\een
Let us
summarize this argument in the following Proposition.

\begin{proposition}
If $M$ is a weakly homogeneous Riemannian manifold,
and if \eqref{E:delta-spec}--\eqref{E:p-range} hold,
then there exists a
minimizer $v \in H^1(M)$ of $F_\lambda(v)$, subject to the constraint
$J_p(v) = \beta$.

\end{proposition}

$\text{}$ \newline
{\bf Remark.} It follows from (\ref{2.2.7}) that convergence $v_\nu\rightarrow
v$ in (\ref{2.2.3A}) holds in norm in $H^1(M)$, hence in norm in $L^{p+1}(M)$.

\subsection{Manifolds with no $F_\lambda$ minimizers}\label{s2.a}

We take
\ben
\lambda>0,\quad \beta>0,
\quad p\in\Bigl(1,1+\frac{4}{n-2}\Bigr),
\label{2.A.1}
\een
and
\ben
M=\Rr^n\setminus K,
\label{2.A.2}
\een
where $K\subset\Rr^n$ is a smoothly bounded, compact set.  We impose the
Dirichlet boundary condition on $\pa M$, pick $\beta>0$, and seek minimizers
of $F_\lambda(u)$, given
\ben
u\in H^1_0(M),\quad J_p(u)=\beta.
\een
We will show that no such minimizer exists.  To see this, first compare
\ben
I_\beta(M)=\inf\, \{F_\lambda(u):u\in H^1_0(M),\, J_p(u)=\beta\}
\label{2.A.4}
\een
with
\ben
I_\beta(\Rr^n)=\inf\, \{F_\lambda(u):u\in H^1(\Rr^n),\, J_p(u)=\beta\}.
\een
It is clear that $I_\beta(M)\ge I_\beta(\Rr^n)$, since (\ref{2.A.4}) is an
inf over a smaller set of functions.

On the other hand, one can take an
$F_\lambda$-minimizer for $\Rr^n$, whose existence is classical (and follows
as a special case from \S{\ref{s2.2}}), or, without appealing to this
existence result, simply take an element $\tilde{u}_\ep\in H^1(\Rr^n)$
such that $J_p(\tilde{u}_\ep)=\beta$ and $F_\lambda(\tilde{u}_\ep)\geq
I_\beta(\Rr^n)-\ep/2$.
Translate this function to be concentrated
far away from $K$, and apply a cutoff to get an element $u_\ep\in H^1_0(M)$
such that $J_p(u_\ep)=\beta$ and $F_\lambda(u_\ep)\geq I_\beta(\Rr^n)-\ep$.
Thus
\ben
I_\beta(M)=I_\beta(\Rr^n).
\een
We now prove the following.

\begin{proposition} \label{p2a.1}
There does not exist $u\in H^1_0(M)$ such that $J_p(u)=\beta$ and
$F_\lambda(u)=I_\beta(M)$.
\end{proposition}

\begin{proof}  Suppose such $u\in H^1_0(M)$ does exist.  We can arrange that
$u\ge 0$ on $M$.  Set $v=u$ on $M$, $v=0$ on $K$.  Then
\ben
v\in H^1(\Rr^n),\quad J_p(v)=\beta,\quad F_\lambda(v)=I_\beta(M)=
I_\beta(\Rr^n),
\een
so $v$ must be an $F_\lambda$-minimizer on $\Rr^n$.
It is well known in the case of
$\Rr^n$ (cf.~\S{\ref{s2.3}} for more general results) that such $v$ must
be $>0$ everywhere on $\Rr^n$, which presents a contradiction.
\end{proof}

$\text{}$\newline
{\bf Remark.} The $F_\lambda$-minimizing sequences described above exhibit
concentration, consistent with results of \S{\ref{s2.1}}.  The lack of an
adequate family of isometries of $M$ in this setting prevents this from
yielding a compactness result, and hence an $F_\lambda$-minimizer.

$\text{}$ \newline
{\bf Remark.} One readily obtains similar non-existence results for the
complement of a compact set in a general noncompact, connected, weakly
homogeneous space.

$\text{}$ \newline
{\bf Remark.} In Appendix \ref{sa6}, we will build on the examples here
to give examples of positive solutions to (\ref{E:SNLS}) that are not
$F_\lambda$-minimizers, and examples of compact manifolds (with boundary)
for which (\ref{E:SNLS}) has two geometrically inequivalent positive
solutions.

\subsection{Smoothness of minimizers}\label{s2.3}

As stated above, in \S\S{\ref{s2.3}--\ref{s2.5}} we assume for simplicity
that $\pa M=\emptyset$.
We begin with a local regularity result.  Let $\Omega \subset M$ be an
open set, and assume $u \in \Hloc^1( \Omega )$ solves
\begin{equation}
\label{E:local-ell}
-\Delta u +\lambda u = f(u), \quad f(u) = K | u |^{p-1} u,
\end{equation}
with $p$ as in \eqref{E:p-range}.

\begin{proposition}
\label{P:local-reg-prop}
Every solution $u \in \Hloc^1( \Omega)$ to \eqref{E:local-ell}
satisfies
\ben
u \in C^{p+2}(\Omega)
\een
if $p \notin \mathbb{N}$.  If $p$ is an odd integer, then $u \in
C^\infty(\Omega)$.  If $p$ is an even integer, $u \in C^s(\Omega)$ for
all $s < p+2$, and if $p$ is an even integer and $u \geq 0$ on
$\Omega$ then $u \in C^\infty(\Omega)$.  Finally, for any $p$
satisfying \eqref{E:p-range}, if $u$ is nowhere vanishing on $\Omega$,
then $u \in C^\infty(\Omega)$.

\end{proposition}

\begin{proof}
This is a standard result, but we sketch the proof here in preparation
for the global results in the sequel.  Recall that for {\it linear}
equations, we have elliptic regularity: for $1<q<\infty$ and $s \geq
0$, 
\ben
\aligned
-\Delta u + \lambda u = f \in \Hloc^{s,q}(\Omega) & \implies u \in
\Hloc^{s+2,q}(\Omega),  \\
-\Delta u + \lambda u = f \in C^{s}(\Omega) & \implies u \in
C^{s+2}(\Omega) \text{ (if } s \notin \ZZ\text{)}. 
\endaligned
\label{E:local-ell-reg}
\een
We also have Sobolev embeddings, such as
\ben
\aligned
& \Hloc^{s,q}(\Omega) \subset \Lloc^{nq/(n-sq)}(\Omega), \quad 0 < s <
\frac{n}{q},  \\
& \Hloc^{s + n/q,q}(\Omega) \subset C^s(\Omega), \quad 0 < s < 1.
\endaligned
\label{E:loc-Sobolev}
\een

In order to get started, write
\ben
p = \frac{1}{\gamma} \frac{n+2}{n-2}
\een
for some $\gamma>1.$  Then for $f(u)$ as in \eqref{E:local-ell},
\ben
\aligned
u \in \Hloc^1(\Omega) & \implies u \in \Lloc^{2n/(n-2)} ( \Omega)
\quad (\text{if } n \geq 3) \\
& \implies f(u) \in \Lloc^{2n\gamma/(n+2)} ( \Omega),
\endaligned
\een
so that \eqref{E:local-ell-reg} yields
\ben
u \in \Hloc^{2, 2n\gamma/(n+2)} (\Omega)
\een
if $n \geq 3$.  
If $n = 2$, then $u \in \Hloc^{2,q}(\Omega)$ for all $q<\infty$.
Observe 
\begin{equation}
\label{E:gamma-n}
\frac{2n \gamma}{n + 2} > \frac{n}{2} \Leftrightarrow \gamma > \frac{n+2}{4}.
\end{equation}
If \eqref{E:gamma-n} holds, we have
\begin{equation}
\label{E:u-in-Cs}
u \in C^s(\Omega)
\end{equation}
for some $s \in (0,1)$.  In the endpoint case $\gamma = (n+2)/4$, we
have
\ben
u \in \Lloc^q ( \Omega), \quad \forall q < \infty,
\een
and hence $f(u) \in \Lloc^q(\Omega)$ for all $q < \infty$ as well.
Then by \eqref{E:local-ell-reg}, 
\ben
u \in \Hloc^{2,q}(\Omega), \quad \forall q < \infty,
\een
and \eqref{E:u-in-Cs} holds in this case as well.  If $\gamma <
(n+2)/4$, we use the Sobolev embeddings \eqref{E:loc-Sobolev} to get
\ben
u \in \Lloc^{2n\gamma/(n+2-4\gamma)} (\Omega),
\een
and hence
\ben
f(u) \in \Lloc^{2n \gamma_2/(n+2)} (\Omega),
\een
where
\ben
\gamma_2 = \gamma^2 \frac{n-2}{n+2-4\gamma} > \gamma^2.
\een
Inserting this improved regularity for $f(u)$ into the elliptic
regularity estimates \eqref{E:local-ell-reg} yields now
\ben
u \in \Hloc^{2, 2n \gamma_2/(n+2)} (\Omega).
\een
A finite number of iterations of this procedure yields the property
\eqref{E:u-in-Cs}.  This in turn implies $f(u) \in C^s(\Omega)$, hence
$u \in C^{s+2} (\Omega)$ for some $s \in (0,1)$.  

From this, the conclusions of Proposition \ref{P:local-reg-prop}
follow, once we observe that if $p>1$ is not an odd integer, then one
cannot improve the implication 
\ben
u \in C^s(\Omega), \,\,s \geq p,
\implies f(u) \in C^p (\Omega),
\een
except when $p$ is an even integer and $u$ does not change sign, while
if $p$ is an odd integer, we get $f(u) \in C^s(\Omega)$.

\end{proof}

For the rest of this section, we assume $M$ satisfies the weak
homogeneity hypothesis.  We want global estimates for functions $u \in
H^1(M)$ satisfying \eqref{E:local-ell} on all of $M$.
We always have
\ben
(\lambda - \Delta )^{-1} : H^{s,q}(M) \to H^{s+2,q}(M)
\een
whenever $\lambda > -\delta_0$ and $q = 2$, however when $q \neq 2$
one often needs a stronger bound on $\lambda$.  Hence we will take a
different approach, which will also yield some decay estimates on
solutions.  

Let $\Omega \subset M$ be a bounded open set, $\tOmega \Subset
\Omega$.  Assume $\tOmega$ contains a ball of radius $D+1$.  Then we
can choose isometries $g_j \in G$ such that, if we set $\tOmega_j =
g_j(\tOmega)$, then the countable collection $\{ \tOmega_j \}$ covers
$M$, and we can assume (since $M$ has bounded geometry) that there
exists $m < \infty$ such that each point $x \in M$ is in at most $m$
of the $\tOmega_j$.  

Now depending on $p$ satisfying \eqref{E:p-range}, let $L$ be the
number of iterations required in the proof of Proposition
\ref{P:local-reg-prop}.  Choose
intermediate nested open sets:
\ben
\tOmega \Subset \Omega^{(L)} \Subset \cdots \Subset \Omega^{(1)}
\Subset \Omega,
\label{2.3.18}
\een
along with the associated translates 
\ben
\Omega^{(\ell)}_j = g_j ( \Omega^{(\ell)} ).
\label{2.3.19}
\een
We set
\begin{equation}
\label{E:Aj}
A_j = \| u \|_{H^1(\Omega_j)},
\end{equation}
so that
\begin{equation}
\label{E:H1-norm-equiv}
\| u \|_{H^1(M) }^2 \simeq \sum_j A_j^2.
\end{equation}
Applying the proof of Proposition \ref{P:local-reg-prop} on $\Omega$
and its translates by isometries yields a similar statement on each
$\Omega_j$ with constants independent of $j$.  From \eqref{E:Aj}, we
have
\ben
\| u \|_{L^{2n/(n-2)}(\Omega_j )} \leq C_1 A_j
\een
with $C_1$ independent of $j$.  As usual, if $n = 2$, this holds for
$L^q(\Omega_j)$, $q < \infty$.  Hence
\ben
\aligned
\| f(u) \|_{L^{2n\gamma / (n+2)}(\Omega_j ) } & \leq C_2 \| u^p
\|_{L^{2n/(n-2)p} (\Omega_j ) } \\
& \leq C_2 C_1^p A_j^p.
\endaligned
\een
The local elliptic regularity then gives
\begin{equation}
\label{E:C4}
\| u \|_{H^{2, 2n\gamma/(n+2)} (\Omega_j^{(1)} )} \leq C_3 (C_2 C_1^p
  A_j^p + C_1 A_j ) \leq C_4 A_j.
\end{equation}
Iterating this argument $L$ times, we obtain
\begin{equation}
\label{E:local-reg-j}
\| u \|_{C^s(\tOmega_j )} \leq C_\star A_j,
\end{equation}
for a constant $C_\star$ independent of $j$, and where $s$ satisfies
similar properties to that in Proposition \ref{P:local-reg-prop}.  We
record the result in the following Proposition.

\begin{proposition}

\label{P:local-reg-j}
If $u \in H^1(M)$ is a solution to \eqref{E:local-ell}, then
\eqref{E:local-reg-j} holds with $A_j$ given by \eqref{E:Aj},
$C_\star$ independent of $j$, and $s = 2+p$ if $p \notin 2
\mathbb{N}$, and for every $s < 2 + p$ if $p \in 2 \mathbb{N}$.
\end{proposition}

\begin{remark}
Observe that the second inequality in \eqref{E:C4} uses that $A_j \leq
\| u \|_{H^1(M)}$ for each $j$.  Hence $C_\star$ is independent of
$j$, but depends on $\| u \|_{H^1(M)}$ in a nonlinear fashion.  In
light of \eqref{E:H1-norm-equiv}, we conclude
\begin{equation}
\label{E:local-reg-sum}
\sum_j \| u \|_{C^s(\tOmega_j ) }^2 \leq C \left( \| u \|_{H^1(M)}
\right) \| u \|_{H^1(M)}^2.
\end{equation}
\end{remark}

\subsection{Positivity of Minimizers}\label{s2.4}

In this subsection we examine the question of positivity of
minimizers.
If $u \in H^1(M)$ is a minimizer of
$F_\lambda$, subject to the constraint $J_p(u) = \beta$, set $v(x) = |
u(x) |$.  Of course we have
\ben
\| v \|_{L^2} = \| u \|_{L^2}, \quad \| \nabla v \|_{L^2} = \| \nabla
u \|_{L^2}, \quad \| v \|_{L^{p+1}} = \| u \|_{L^{p+1}},
\een
so $v$ is a solution to the same constrained minimization problem.
But then
\ben
v \geq 0, v \in H^1(M), \quad -\Delta v + \lambda v = K_0 | v |^{p-1}
v,
\label{2.4.2}
\een
with $K_0 = I_\beta / \beta$ as in \eqref{E:K0}.   Then Proposition
\ref{P:local-reg-prop} implies $v \in C^{2 + p }(M)$, and $v \in
C^\infty(M)$ if $p$ is an integer.  We improve this in the next
Proposition.

\begin{proposition}\label{p2.4.1}
\label{P:positivity}
The function $v = | u(x) |$ satisfies
\ben
v(x) >0
\een
for all $x \in M$, and hence $v \in C^\infty(M)$.
\end{proposition}

This result is a consequence of the following Harnack inequality, which
follows from Theorem 8.20 -- Corollary 8.21 of \cite{GiTr-book}.
If $v\ge 0$ solves (\ref{2.4.2}) on the weakly homogeneous space $M$, and
if $\Omega^{(\ell)}_j$ are as in (\ref{2.3.18})--(\ref{2.3.19}), then
(in light of the bounds on $|v|^{p-1}$ established in \S{\ref{s2.3}})
there exists a constant $C_0$, independent of $j$, such that
\ben
\sup_{ x \in \Omega_j^{(1)}} u(x) \leq C_0 \inf_{ x \in
  \Omega_j^{(1)}} u(x),
\label{2.4.4}
\een

Given (\ref{2.4.4}), if $v\ge 0$ solves (\ref{2.4.2}) and is not
$\equiv 0$, strict positivity is immediate, and smoothness follows from
Proposition \ref{P:local-reg-prop}, so Proposition \ref{p2.4.1} is proven.
This in turn immediately gives the following.

\begin{corollary}
Every real-valued $F_\lambda$-minimizer $u$ satisfies either $u >0$ on $M$
or $u <0$ on $M$.
\end{corollary}

We now wish to extend the global regularity estimates on $u$ beyond
$s = 2 + p$, when $u>0$ on $M$.
The issue is that although $u>0$ on $M$, $u$ must decay at infinity, and
since $f(u)$ is singular at $u = 0$, there is some work to be done.
Again the Harnack inequality (\ref{2.4.4}) (with $u$ in place of $v$)
provides the key to success.  With this, we can establish the following
improvement of Proposition \ref{P:local-reg-j}.

\begin{proposition}\label{p2.4.4}
If $u \in H^1(M)$ solves \eqref{E:local-ell} on $M$ and $u>0$, then  
\eqref{E:local-reg-j} and \eqref{E:local-reg-sum} hold for every $s <
\infty$ with constants depending on $s$ but not on $j$.

\end{proposition}

\begin{proof}

It suffices to prove the statment for $s \in \mathbb{N}$.  Proposition
\ref{P:local-reg-j} implies that the assertions are true when $s = 3$,
and we proceed by induction.  Let $k \in \NN$ and suppose
\eqref{E:local-reg-j} holds for $s = k$.  Then the covering property
of the $\{ \tOmega_j \}$ implies
\ben
\| u \|_{C^k(\Omega_j)} \leq C_k A_j
\een
for some $C_k$ independent of $j$.  We want to show 
\begin{equation}
\label{E:u-in-Ck1}
\| u \|_{C^{k+1}(\tOmega_j)} \leq C_{k+1} A_j
\end{equation}
for some $C_{k+1}$ independent of $j$.

We need to estimate the $C^k$ norm of $f(u)$.  The chain rule applied
to $f(u)$ gives
\ben
D^\alpha f(u) = \sum_{\gamma_1 + \gamma_2 \cdots + \gamma_\nu =
  \alpha,  \nu , | \gamma_\mu| \geq 1 } C_\gamma u^{(\gamma_1)} \cdots
u^{(\gamma_\nu)} f^{(\nu)} (u).
\een
Now
\ben
| f^{(\nu)}(u) | \leq C | u |^{p-\nu},
\een
and from the Harnack inequality and our induction hypothesis, there
exists $C_1, C_2$ such that 
\ben
C_1 A_j \leq u \leq C_2 A_j \text{ on } \Omega_j^{(1)},
\een
so that
\ben
| f^{(\nu)}| \leq C A_j^{p - \nu}.
\een
Hence for $| \alpha | \leq k$ (so that, in particular, $ \nu ,
| \gamma_\mu| \leq k$), we have
\ben
\aligned
| D^\alpha f(u) | & \leq C \sum_\gamma A_j^\nu A_j^{p-\nu} \\
& \leq C' A_j^p \\
& \leq C'' A_j,
\endaligned
\label{2.4.13}
\een
where again $C''$ is allowed to depend nonlinearly on the quantity $\|
u \|_{H^1(M)}$.  The last inequality in (\ref{2.4.13}) uses the global
bound $A_j\le \|u\|_{H^1(M)}$.
Hence
\ben
\| f(u) \|_{C^k(\Omega_j^{(1)})} \leq C A_j.
\een
Then the local elliptic regularity applied to \eqref{E:local-ell}
yields \eqref{E:u-in-Ck1}, completing the proof.

\end{proof}

\subsection{Further decay estimates}\label{s2.5}

In this section we continue to study properties of a positive solution
$u\in H^1(M)$ to the elliptic equation
\ben
-\Delta u + \lambda u = f(u),
\een
where
\ben
f(u) = K | u |^{p-1} u.
\een
We also continue to assume this equation holds on a manifold $M$
that is weakly homogeneous in the sense described in the previous
sections.  So far we have shown that $u \in L^q(M)$ for every $q \in
[2,\infty]$, and for each $s < \infty$,
\ben
\|u\|_{C^s(\tOmega_j)}\le C_s A_j,\quad \forall\, j,
\label{2.5.2A}
\een
where $\{ \tOmega_j \}$ is an open
cover of $M$ by sets which are images under isometries of a fixed set
$\tOmega$ and $A_j = \| u \|_{H^1(\Omega_j)}$, so
\ben
\sum\limits_j A^2_j\approx \|u\|^2_{H^1(M)}.
\label{2.5.2B}
\een
These are varieties of decay results.  In this section we seek stronger
decay results.  Here, we replace the hypothesis $\lambda>-\delta_0$ by
\ben
\lambda>0,
\een
which for $\delta_0=0$ involves no change.

Since $\{e^{t\Delta}: t \geq 0 \}$ is a
contraction semigroup on $L^q(M)$ for each $q \geq 1$, we have
\ben
(-\Delta + \lambda)^{-1} = \int_0^\infty e^{-\lambda t} e^{t \Delta}
dt,
\een
which implies
\ben
(-\Delta + \lambda)^{-1} : L^q(M) \to L^q(M)
\een
for every $q \in [1,\infty]$, with operator norm bounded by
$\lambda^{-1}$.  Our previously estabished $L^q$ estimates on $u$ imply
\ben
f(u) \in L^q(M),\quad \forall\, q \in [1, \infty].
\een
Since $u = (-\Delta + \lambda)^{-1} f(u)$, we hence have $u \in
L^q(M)$ for every $q \in [1,\infty]$.  

Now set 
\begin{equation}
\label{E:Bj-def}
B_j = \| u \|_{L^1(\Omega_j ) }
\end{equation}
so that
\ben
\sum_j B_j \simeq \|  u \|_{L^1(M)}.
\een
Comparing to (\ref{2.5.2B}), we see the collection $\{ B_j
\}$ satisfy ``stricter bounds'' than $\{A_j\}$.
In this sense, the following result improves (\ref{2.5.2A}).

\begin{proposition}\label{p2.5.1}
For each $\epsilon \in (0,1)$, and each $s<\infty$, there exists
$C_{\epsilon, s}< \infty$ such that
\ben
\| u \|_{C^s(\tOmega_j)} \leq C_{\epsilon,s} A_j^\epsilon
B_j^{1-\epsilon},\quad \forall\, j.
\een

\end{proposition}

\begin{proof}
To start, (\ref{2.5.2A}) implies
\ben
\| u \|_{L^\infty(\tOmega_j)} \leq C A_j
\een
for every $j$, with constants independent of $j$.  Interpolating with
\eqref{E:Bj-def}, for each $\epsilon>0$ we can produce $q>1$ such that
\begin{equation}
\label{E:u-first-int}
\| u \|_{L^q(\tOmega_j)} \leq C A_j^{\epsilon/2} B_j^{1 - \epsilon/2},
\end{equation}
for every $j$, with constants indepenent of $j$.  Next,
(\ref{2.5.2A}) implies that for each $k<\infty$ there exists a
constant $C_k<\infty$ such that
\ben
\|  u \|_{H^{k,q}(\tOmega_j)} \leq C_k A_j,
\een
for every $j$, where again $C_k$ is independent of $j$.  Then
interpolation with \eqref{E:u-first-int} implies that, for each
$\epsilon \in (0,1)$, $\sigma < \infty,$ there exists $C_{\epsilon,
  \sigma}$ such that
\ben
\| u \|_{H^{\sigma, q}(\tOmega_j)} \leq C_{\epsilon, \sigma}
A_j^\epsilon B_j^{1-\epsilon},
\een
where again the constants are independent of $j$.  Taking $\sigma>0$
sufficiently large proves the Proposition.

\end{proof}

\section{Energy Minimizers}\label{s3}

In this section we tackle a different constrained
minimization scheme, albeit for a slightly smaller range of powers $p$
in the nonlinear term of \eqref{E:NLS}.  The constrained minimization
procedure is one that minimizes energy with respect to fixed mass.
As stated in \S{\ref{s1}}, we now require
\begin{eqnarray}
\label{E:subcrit}
p \in \Bigl(1, 1 + \frac{4}{n}\Bigr),
\end{eqnarray}
which is the range of $L^2$ subcritical powers in the standard Euclidean case
example.  Such a case was also handled using concentration compactness on
$\Rr^n$ in \cite{L1}, which we here generalize to the setting of weakly
homogeneous spaces, $M$ as defined in \S{\ref{s2.2}}.  We desire to minimize
the functional
\begin{eqnarray}
E(u) = \frac12 \| \nabla u \|_{L^2}^2 - \frac{1}{p+1} \int_M |u|^{p+1}\, dV
\label{3.0.2}
\end{eqnarray}
over $H^1(M)$, subject to the constraint
\begin{eqnarray}
Q(u) = \| u \|_{L^2}^2 = \beta.
\label{3.0.3}
\end{eqnarray}
As in \S{\ref{s2}}, $H^1(M)$ will stand for $H^1_0(M)$ if $\pa M\neq\emptyset$
and we use the Dirichlet boundary condition.
As seen in (\ref{1.2.1})--(\ref{1.2.3}), given $u \in H^1 (M)$ a
solution to the constrained minimization problem, we must have,
for some $\lambda\in\Rr$,
\begin{eqnarray}
-\Delta u + \lambda u - |u|^{p-1} u = 0.
\end{eqnarray}

The range of powers in (\ref{E:subcrit})
plays an important role in the Gagliardo-Nirenberg
inequality
\begin{eqnarray}
\| u \|_{L^{p+1} (\RR^n)} \leq C \| u \|_{L^2}^{1-\gamma}
\| u \|_{H^1}^{\gamma},
\label{3.0.5}
\end{eqnarray}
where
\begin{eqnarray}
\gamma = \frac{n}{2} - \frac{n}{p-1}, \quad \text{hence }\
\gamma(p+1)<2.
\end{eqnarray}
As a result, we have
\ben
\aligned
\| u \|_{H^1}^2 &= E(u) + \frac{1}{p+1} \| u \|_{L^{p+1}}^{p+1} + Q(u) \\
& \leq E(u) + \tilde{C} Q(u)^{{(p+1)(1-\gamma)}/{2}}
\| u \|_{H^1}^{\gamma (p+1)} + Q(u), 
\endaligned
\label{E:EQbds}
\een
which gives a priori bounds at $\| u \|_{H^1}$ in terms of bounds on $E(u)$
and $Q(u)$.  Then, as in (\ref{1.3.2}), we take 
\begin{equation}
\label{E:I-beta-Emin}
\Cal{I}_\beta =  \inf \{ E (u)  :  u \in H^1(M), \, Q (u) = \beta \},
\end{equation}
for $\beta > 0$.
The a priori bounds in \eqref{E:EQbds} show that for a
particular $\beta$, we have $\Cal{I}_\beta > -\infty$ since $\gamma(p+1) < 2$.
Taking at this point the sequence $u_\nu \in H^1(M)$ such that
\begin{eqnarray}
Q(u_\nu) = \beta, \ E(u_\nu) \leq \Cal{I}_\beta + \frac1{\nu}.
\label{3.0.9}
\end{eqnarray}
Note, from \eqref{E:EQbds}, we have that $\| u_\nu \|_{H^1}$ bounded.
Then, as in \S{\ref{s2}}, we apply the concentration-compactness techniques 
to the $L^1$ sequence given by $\{ | u_\nu |^2 \}$.

It turns out that we need to assume
\ben
\Cal{I}_\beta < 0,
\label{3.0.9A}
\een
to show in the concentration-compactness argument
that splitting and vanishing cannot occur.  In connection with this,
note that replacing $u$ by $au$ in (\ref{3.0.2}) and letting $a\nearrow
+\infty$ shows that
\ben
\Cal{I}_\beta \to -\infty \ \text{ as }\ \beta \to +\infty.
\label{E:i-beta}
\een
In particular, $\Cal{I}_\beta<0$ for all sufficiently large $\beta$.
However, it is not guaranteed that $\Cal{I}_\beta<0$ for all $\beta$.
See Appendix \ref{sa2} for more on this.
Exploration of when such a negative energy
condition is satisfied for a weakly homogeneous space is an interesting area
for future research.

In \S{\ref{s3.1}} we demonstrate concentration for a subsequence of a
minimizing sequence (\ref{3.0.9}), when (\ref{3.0.9A}) holds.  In
\S{\ref{s3.2}} we establish compactness and prove existence of energy
minimizers when $M$ is a weakly homogeneous space, again under the
hypothesis (\ref{3.0.9A}).  (If $M$ is {\it compact}, (\ref{3.0.9A}) is not
needed.)  In \S{\ref{s3.3}} we show that the examples of manifolds with
no $F_\lambda$-minimizers given in \S{\ref{s2.a}} also have no energy
minimizers.  In \S{\ref{s3.4}} we note how results on smoothness,
positivity, and decay established for $F_\lambda$-minimizers in
\S\S{\ref{s2.3}--\ref{s2.5}} also hold for energy minimizers.
In \S{\ref{s3.5}} we compute the second variation of energy for an
energy minimizer, expressed in terms of operators $L_{\pm}$, defined in
(\ref{3.5.44}).  In \S{\ref{s3.a}} we give similar formulas for the
second variation of $F_\lambda$, for $F_\lambda$-minimizers.
In \S{\ref{s3.6}} we examine some spectral properties of
$L_{\pm}$, and draw a number of conclusions.  In particular, we
deduce from the fact that $L_-\geq 0$ that whenever an
energy minimizer $u\in H^1(M)$ satisfies (3.0.4), $\Spec (-\Delta+\lambda)
\subset [0,\infty)$.  Results of \S{\ref{s3.6}} are applied in
\S{\ref{s3.8}} to results concerning orbital stability.

\subsection{Concentration}\label{s3.1}

We need to show that there is no vanishing and no splitting.  We first
establish that there is no vanishing when $\Cal{I}_\beta<0$.  In fact,
The vanishing condition for our sequence implies that given $B_R (y) =
\{ x \in M : d_M (x,y) \leq R \}$, we have
\begin{eqnarray}
\label{E:van}
\lim_{\nu \to \infty} \sup_{y \in M} \int_{B_R (y) } | u_\nu |^2\, dV = 0,
\quad \forall\, R < \infty.
\end{eqnarray}
Using Lemma \ref{L:vanishing}, we have the following.

\begin{proposition}
\label{P:van}
Assume $\Cal{I}_\beta<0$.
For $u_\nu$ a sequence minimizing the energy with fixed mass, \eqref{E:van}
cannot occur.
\end{proposition}

\begin{proof}
For $p$ as in \eqref{E:subcrit}, let us assume that \eqref{E:van} occurs for
$\{ u_\nu \}$.  Then, Lemma \ref{L:vanishing} shows
\begin{eqnarray}
\| u_\nu \|_{L^{p+1}} \to 0,
\end{eqnarray}
implying
\begin{eqnarray}
\frac12 \| \nabla u_\nu \|_{L^2}^2 \to \Cal{I}_\beta < 0,
\end{eqnarray}
a contradiction.
\end{proof}

Our next task is to establish that there is no splitting.
If there were splitting, we see that for any $\alpha \in (0,\beta)$, for each
$\epsilon > 0$ there exists $\nu_0 \geq 1$ and sets $E_\nu^\#, \ E_\nu^b
\subset M$ such that
\begin{eqnarray}
d(E_\nu^\#,E_\nu^b) \to \infty \ \text{as} \ \nu \to \infty
\end{eqnarray}
and
\begin{eqnarray}
\left| \int_{E^\#_\nu} |u_\nu|^2\, dV - \alpha \right| < \epsilon,
\ \ \left| \int_{E^b_\nu} |u_\nu|^2\, dV - (\beta - \alpha) \right| < \epsilon .
\end{eqnarray}
We record here some subadditivity properties of $\Cal{I}_\beta$ in order to
argue similarly to the splitting argument in Section \ref{S:conc}.

\begin{proposition}
\label{P:subadd0}
If $\beta >0$, $\Cal{I}_\beta < 0$, $\sigma > 1$, then
\begin{eqnarray}
\Cal{I}_{\sigma \beta} < \sigma \Cal{I}_\beta.
\end{eqnarray}
\end{proposition}

\begin{proof}
Let $u_\nu$ be a minimizing sequence as in (\ref{3.0.9}).  Define
\begin{eqnarray}
w_\nu = \sigma^{1/2} u_\nu,\quad \text{so }\
\| w_\nu \|_{L^2}^2 = \sigma \beta.
\end{eqnarray}
Hence,
\ben
\aligned
E(w_\nu)
&=\frac{\sigma}{2}\|\nabla u_\nu\|^2_{L^2}
-\frac{\sigma^{p+1}}{p+1} \int_M |u_\nu|^{p+1}\, dV \\
&= \sigma E( u_\nu) - \frac{\sigma^{p+1} - \sigma}{p+1}
\| u_\nu \|_{L^{p+1}}^{p+1}.
\endaligned
\een
Passing to the limit gives
\begin{eqnarray}
\Cal{I}_{\sigma \beta} \leq \sigma \Cal{I}_{\beta}.
\end{eqnarray}
However, given $\Cal{I}_\beta<0$, 
as in the proof of Proposition \ref{P:van}, $\| u_\nu \|_{L^{p+1}} $
does not approach $0$ and the result follows.
\end{proof}

Then, we have the following result similar to Proposition \ref{P:subadd}. 

\begin{proposition}
\label{P:subadd1}
Given $0 < \eta < \beta$ and $\Cal{I}_\beta < 0$, we have
\begin{eqnarray}
\Cal{I}_\beta < \Cal{I}_{\beta - \eta} + \Cal{I}_\eta.
\end{eqnarray}
\end{proposition}

\begin{proof}
Let us assume without loss of generality that $\eta \leq \beta - \eta$ and
take
\begin{eqnarray}
\beta - \eta = \sigma \eta
\end{eqnarray}
with $\sigma \geq 1$.
Hence, using Proposition \ref{P:subadd0} we have
\begin{eqnarray}
\Cal{I}_{\sigma \eta} \leq \sigma \Cal{I}_\eta, \
\Cal{I}_\beta = \Cal{I}_{(\sigma + 1) \eta} <
\frac{ \sigma +1 }{\sigma} \Cal{I}_{\sigma \eta} = \frac{ \sigma +1 }{\sigma}
\Cal{I}_{\beta-\eta}.
\end{eqnarray}
As a result,
\ben
\aligned
\Cal{I}_\beta = \Cal{I}_{ (\sigma + 1) \eta} & < \frac{ \sigma +1 }{\sigma}
\Cal{I}_{\sigma \eta} \\
& = \Cal{I}_{\beta - \eta} + \frac{1}{\sigma} \Cal{I}_{\sigma \eta} \\
& \leq \Cal{I}_{\beta - \eta} + \Cal{I}_\eta.
\endaligned
\een
\end{proof}
Applying this proposition in the same way as in Section \ref{S:conc},
we have the result that no splitting can occur for the sequence
$\{ u_\nu \}$.

Therefore, upon passing to a subsequence,
we have concentration.  There exist $y_\nu\in M$ (independent
of $\epsilon$) with the following property.  For each $\epsilon>0$, there
exists $\widetilde{R}(\epsilon)<\infty$ such that
\ben
\int_{B_{\widetilde{R}(\epsilon)}(y_\nu)} |u_\nu|^2\, dV>\beta-\epsilon.
\label{3.1.16}
\een

\subsection{Existence of energy minimizers}\label{s3.2}

As long as
\ben
\Cal{I}_{\beta}<0,
\label{4.0}
\een
we are left with the situation where $u_\nu\in H^1(M)$ 
satisfies (\ref{3.0.9})
and the concentration phenomenon (\ref{3.1.16}).
If $M$ is weakly homogeneous, we can translate the points $y_\nu$ to a subset
of some compact $K\subset M$.  We relabel the associated translates of
$u_\nu$ as $u_\nu$.  Passing to a subsequence, we have
\ben
u_\nu\longrightarrow u,\quad \text{weak}^*\ \text{in}\ H^1(M).
\label{4.1}
\een
By Rellich compactness, $u_\nu\rightarrow u$ in $L^2(B)$, in norm, for each
bounded $B\subset M$.  Hence, by (\ref{3.1.16})
\ben
\|u\|^2_{L^2}=\beta.
\label{4.2}
\een
Hence,
\ben
u_\nu\longrightarrow u\quad \text{in }\ L^2(M)\ \text{norm}.
\label{4.3}
\een
Now, as in (\ref{3.0.5}), we have
\ben
\|u-u_\nu\|_{L^{p+1}}\le C\|u-u_\nu\|_{L^2}^{1-\gamma}
\|u-u_\nu\|_{H^1}^\gamma,
\label{4.4}
\een
so
\ben
u_\nu\longrightarrow u\ \text{ in }\ L^{p+1}(M)\ \text{norm}.
\label{4.5}
\een
Now
\ben
\frac{1}{2}\|\nabla u_\nu\|^2_{L^2}-\frac{1}{p+1}\|u_\nu\|^{p+1}_{L^{p+1}}
\longrightarrow \Cal{I}_\beta,
\label{4.6}
\een
and
\ben
\|u_\nu\|^{p+1}_{L^{p+1}}\longrightarrow \|u\|^{p+1}_{L^{p+1}}.
\label{4.7}
\een
Also, since (\ref{4.2}) holds,
\ben
\frac{1}{2}\|\nabla u\|^2_{L^2}-\frac{1}{p+1}\|u\|^{p+1}_{L^{p+1}}
\ge \Cal{I}_\beta.
\label{4.8}
\een
Hence
\ben
\|\nabla u\|^2_{L^2}\ge \liminf \|\nabla u_\nu\|^2_{L^2},
\label{4.9}
\een
so
\ben
\nabla u_\nu\longrightarrow \nabla u\ \text{ in }\ L^2(M)\ \text{norm,}
\label{4.10}
\een
and $u$ minimizes $E(u)$ subject to the constraint (\ref{3.0.3}), at least
provided (\ref{4.0}) holds.

Here is one basic case where the hypothesis (\ref{4.0}) can be removed.

\begin{proposition}\label{p3.2.1}
Let $M$ be a compact $n$-dimensional Riemannian manifold, possibly with
boundary, and assume $p$ satisfies (\ref{E:subcrit}).  Then, given
$\beta>0$, there exists $u\in H^1(M)$ such that $Q(u)=\beta$ and
$E(u)=\Cal{I}_\beta$.
\end{proposition}

\begin{proof}
Taking $u_\nu$ as in (\ref{3.0.9}), we have a bound on $\|u_\nu\|_{H^1}$.
If $u_\nu\rightarrow u$ $\text{weak}^*$ in $H^1(M)$, the Rellich compactness
theorem yields $u_\nu\rightarrow u$ in norm in both $L^2(M)$ and
$L^{p+1}(M)$.  Hence $Q(u)=\beta$ and $\|u\|_{L^{p+1}}=\lim\,
\|u_\nu\|_{L^{p+1}}$, hence $E(u)\le \lim\, E(u_\nu)$.  This implies
$E(u)=\Cal{I}_\beta$ (and also $u_\nu\rightarrow u$ in $H^1$-norm).
\end{proof}

\subsection{Manifolds with no energy minimizers}\label{s3.3}

Here we show that the manifolds with no $F_\lambda$-minimizers exhibited
in \S{\ref{s2.a}} also have no energy minimizers.
Since the arguments are similar,
we will be brief.  Let $M=\Rr^n\setminus K$, where $K\subset\Rr^n$ is a
smoothly bounded, compact set. We impose the Dirichlet boundary condition
on $\pa M$, take $\beta>0$, and ask whether we can minimize $E(u)$, given
\ben
u\in H^1_0(M),\quad Q(u)=\beta.
\een
We will show that no such minimizer exists.  To see this, set
\ben
\Cal{I}_\beta(M)=\inf\, \{E(u):u\in H^1_0(M),\, Q(u)=\beta\}.
\een
It is clear that $\Cal{I}_\beta(M)\ge \Cal{I}_\beta(\Rr^n)$, since
$H^1_0(M)\subset H^1(\Rr^n)$.  An argument similar to that in \S{\ref{s2.a}}
yields the reverse inequality, so
\ben
\Cal{I}_\beta(M)=\Cal{I}_\beta(\Rr^n).
\een
The positivity results of \S{\ref{s2.4}} apply to this setting
(for more on this, see \S{\ref{s3.4}}).
The proof of Proposition \ref{p2a.1} is hence readily modified, to yield:

\begin{proposition}\label{p3.3.1}
If $M=\Rr^n\setminus K$, there does not exist $u\in H^1_0(M)$ such that
$Q(u)=\beta$ and $E(u)=\Cal{I}_\beta(M)$.
\end{proposition}

As in \S{\ref{s2.a}}, we can replace $\Rr^n$ by a general noncompact,
connected, weakly homogeneous space and get a similar nonexistence result.

\subsection{Smoothness, positivity, and decay of energy minimizers}
\label{s3.4}

In this brief section, we make note of how results of
\S\S{\ref{s2.3}--\ref{s2.5}} apply to energy minimizers.  We return to the
setting where $M$ is a weakly homogeneous.
As was done in \S\S{\ref{s2.3}--\ref{s2.5}}, in this section we assume,
for the sake of simplicity, that $\pa M=\emptyset$.

If $u\in H^1(M)$ minimizes (\ref{3.0.2}), subject to the constraint
(\ref{3.0.3}), so does $v=|u|$, so $v$ solves
\ben
v\in H^1(M),\ v\ge 0,\quad -\Delta v+\lambda v-|v|^{p-1}v=0,
\een
for some $\lambda\in\Rr$.  Boundedness (and decay) results of
\S{\ref{s2.3}} hold.  Then, as in \S{\ref{s2.4}}, the Harnack inequality
(\ref{2.4.4}) impliies $v>0$ on $M$, and we get:

\begin{proposition}\label{p3.4.1}
Every real-valued energy minimizer $u$ satisfies either $u>0$ on $M$
or $u<0$ on $M$, and belongs to $\Cal{C}^\infty(M)$.
\end{proposition}

Given this, the decay results Proposition \ref{p2.4.4} and Proposition
\ref{p2.5.1} apply to these energy minimizers.

\subsection{Second variation of energy}\label{s3.5}

With $\beta\in (0,\infty)$, let
\ben
X=\{u\in H^1(M):Q(u)=\beta\},\quad
\Cal{I}_\beta=\inf\,\{E(u):u\in X\},
\label{3.5.1}
\een
and
\ben
Y=\{u\in X:E(u)=\Cal{I}_\beta\}.
\label{3.5.2}
\een
Conditions guaranteeing that $Y$ is nonempty have been given in \S{\ref{s3.2}}.
Recall that $E(u)$ and $Q(u)$ are given by (\ref{3.0.2})--(\ref{3.0.3}).
Here we study
\ben
\frac{d^2}{ds^2}E(w(s)),
\label{3.5.4}
\een
when $w(s)$ is a smooth path in $X$ satisfying
$w(0)=u\in Y$.
To be definite, take $u\in Y$,
\ben
\psi\in T_uX=\{\psi\in H^1(M):\Re (u,\psi)=0\},
\label{3.5.6}
\een
and set
\ben
w(s)=a\, \frac{u+s\psi}{\|u+s\psi\|},\quad a=\beta^{1/2}.
\label{3.5.7}
\een
In light of the discussion in \S{\ref{s3.4}}, we can assume
\ben
u>0\ \text{ on }\ M,
\label{3.5.8}
\een
but we cannot assume $\psi$ is real valued.  Set
\ben
\psi=\psi_0+i\psi_1,\quad \psi_0,\psi_1\ \text{ real valued}.
\label{3.5.9}
\een
Then the condition (\ref{3.5.6}) is equivalent to
\ben
(u,\psi_0)=0,
\label{3.5.10}
\een
with no constraint on $\psi_1$.

The chain rule gives
\ben
\frac{d}{ds}E(w(s))=DE(w(s))w'(s),
\label{3.5.11}
\een
and in particular
\ben
\frac{d}{ds}E(w(s))\Bigr|_{s=0}=DE(u)w'(0).
\label{3.5.12}
\een
Differentiating (\ref{3.5.11}) gives
\ben
\frac{d^2}{ds^2}E(w(s))=D^2E(w(s))(w'(s),w'(s))+DE(w(s))w''(s),
\label{3.5.13}
\een
and in particular
\ben
\frac{d^2}{ds^2}E(w(s))\Bigr|_{s=0}=D^2E(u)(w'(0),w'(0))+DE(u)w''(0).
\label{3.5.14}
\een

We turn to the computation of $w'(0)$ and $w''(0)$.  Rewrite (\ref{3.5.7}) as
\ben
\aligned
w(s)&=aF(s)(u+s\psi), \\
F(s)&=\|u+s\psi\|^{-1}
=(a^2+\|\psi\|^2)^{-1/2}.
\endaligned
\label{3.5.15}
\een
Then
\ben
\aligned
F'(s)&=-s(a^2+s^2\|\psi\|^2)^{-3/2}\|\psi\|^2, \\
F''(s)&=-(a^2+s^2\|\psi\|^2)^{-3/2}\|\psi\|^2
-s\|\psi\|^2\, \frac{d}{ds}(a^2+s^2\|\psi\|^2)^{-3/2},
\endaligned
\label{3.5.16}
\een
so
\ben
F'(0)=0,\quad F''(0)=-a^{-3}\|\psi\|^2.
\label{3.5.17}
\een
We have
\ben
\aligned
w'(s)&=aF(s)\psi+aF'(s)(u+s\psi),\quad \text{so } \\
w'(0)&=aF(0)\psi=\psi,
\endaligned
\label{3.5.18}
\een
and
\ben
\aligned
w''(s)&=2aF'(s)\psi+aF''(s)u,\quad \text{so } \\
w''(0)&=aF''(0)u=-\frac{\|\psi\|^2}{a^2}u.
\endaligned
\label{3.5.20}
\een
Thus (\ref{3.5.12}) and (\ref{3.5.14}) become
\ben
\frac{d}{ds}E(w(s))\Bigr|_{s=0}=DE(u)\psi,
\label{3.5.22}
\een
and
\ben
\frac{d^2}{ds^2}E(w(s))\Bigr|_{s=0}=D^2E(u)(\psi,\psi)-\frac{1}{a^2}
\|\psi\|^2DE(u)u.
\label{3.5.23}
\een
Also
\ben
DE(u)\psi=\frac{d}{ds}E(u+s\psi)\Bigr|_{s=0},
\label{3.5.24}
\een
and
\ben
D^2E(u)(\psi,\psi)=\frac{d^2}{ds^2}E(u+s\psi)\Bigr|_{s=0}.
\label{3.5.25}
\een

Our next task is to compute the right sides of (\ref{3.5.24}) and
(\ref{3.5.25}).  It is convenient to set
\ben
E(u)=T(u)-\tJ(u),
\label{3.5.25A}
\een
with
\ben
T(u)=\frac{1}{2}\|\nabla u\|^2_{L^2},\quad
\tJ(u)=\frac{1}{p+1}\|u\|^{p+1}_{L^{p+1}}.
\label{3.5.25B}
\een
First, the calculation
\ben
\aligned
T(s+s\psi)&=\frac{1}{2}\|\nabla u+s\nabla\psi\|^2 \\
&=\frac{1}{2}\|\nabla u\|^2+s\,\Re (\nabla u,\nabla\psi)+\frac{s^2}{2}
\|\nabla\psi\|^2
\endaligned
\label{3.5.26}
\een
gives
\ben
\frac{d}{ds}T(u+s\psi)\Bigr|_{s=0}=\Re (\nabla u,\nabla\psi)=\Re (-\Delta u,\psi),
\label{3.5.27}
\een
and
\ben
\frac{d^2}{ds^2}T(u+s\psi)\Bigr|_{s=0}=\|\nabla\psi\|^2=\|\nabla\psi_0\|^2
+\|\nabla\psi_1\|^2.
\label{3.5.28}
\een
Next, write
\ben
\tJ(u+s\psi)=\frac{1}{p+1}\int\limits_M (u+s\psi)^{(p+1)/2}(\ubar+s\psibar
)^{(p+1)/2}\, dV.
\label{3.5.29}
\een
Then
\ben
\aligned
\frac{d}{ds}\tJ(u+s\psi)=\frac{1}{2}\int\limits_M\Bigl[
&(u+s\psi)^{(p-1)/2}(\ubar+s\psibar)^{(p+1)/2}\psi \\
&+(u+s\psi)^{(p+1)/2}(\ubar+s\psibar)^{(p-1)/2}\psibar\Bigr]\,dV,
\endaligned
\label{3.5.30}
\een
and
\ben
\aligned
\frac{d^2}{ds^2}\tJ(u+s\psi)=\frac{1}{2}\int\limits_M \Bigl[
&\frac{p+1}{2}(u+s\psi)^{(p-1)/2}(\ubar+s\psibar)^{(p-1)/2}\psi\psibar \\
&+\frac{p-1}{2}(u+s\psi)^{(p-3)/2}(\ubar+s\psibar)^{(p+1)/2}\psi^2 \\
&+\frac{p+1}{2}(u+s\psi)^{(p-1)/2}(\ubar+s\psibar)^{(p-1)/2}\psi\psibar \\
&+\frac{p-1}{2}(u+s\psi)^{(p+1)/2}(\ubar+s\psibar)^{(p-3)/2}\psibar^2
\Bigr]\,dV.
\endaligned
\label{3.5.31}
\een
In particular,
\ben
\aligned
\frac{d}{ds}\tJ(u+s\psi)\Bigr|_{s=0}
&=\frac{1}{2}\int\limits_M |u|^{p-1}(\ubar\psi+u\psibar) \\
&=\Re (|u|^{p-1}u,\psi).
\endaligned
\label{3.5.32}
\een

Before evaluating (\ref{3.5.31}) at $s=0$, let us record that (\ref{3.5.27})
and (\ref{3.5.32}) imply
\ben
\frac{d}{ds}E(u+s\psi)\Bigr|_{s=0}=\Re (-\Delta u-|u|^{p-1}u,\psi).
\label{3.5.33}
\een
(This calculation does not use (\ref{3.5.8}).)
For $u\in Y$, i.e., a minimizer of $E|_X$,
this must vanish for all $\psi\in T_u X$, described by (\ref{3.5.6}).
Consequently, given $u\in Y$,
\ben
\psi\in H^1(M),\ \Re (u,\psi)=0\Longrightarrow \Re (\Delta u+|u|^{p-1}u,\psi)
=0.
\label{3.5.34}
\een
It follows that there exists $\lambda\in\Rr$ such that
\ben
\Delta u+|u|^{p-1}u=\lambda u,
\label{3.5.35}
\een
and we recover (\ref{1.2.2})--(\ref{1.2.3}).

We also note that the last term in (\ref{3.5.23}) is
\ben
\aligned
\frac{1}{a^2}\|\psi\|^2\frac{d}{ds}&E(u+su)\Bigr|_{s=0} \\
&=\frac{1}{a^2}\|\psi\|^2\, \Re(-\Delta u-|u|^{p-1}u,u) \\
&=\frac{1}{a^2}\|\psi\|^2(-\lambda u,u) \\
&=-\lambda \|\psi\|^2,
\endaligned
\label{3.5.36}
\een
the second identity by (\ref{3.5.35}).

We now evaluate (\ref{3.5.31}) at $s=0$.  For this, we will use (\ref{3.5.8}),
and write $\psi=\psi_0+i\psi_1$, as in (\ref{3.5.9}).  We have
\ben
\aligned
\frac{d^2}{ds^2}\tJ(u+s\psi)\Bigr|_{s=0}
&=\frac{1}{2}\int\limits_M \Bigl[(p+1)|u|^{p-1}|\psi|^2 \\
&\ \ \ \ \ \
+\frac{p-1}{2}u^{(p-3)/2}\ubar^{(p+1)/2}\psi^2 \\
&\ \ \ \ \ \
+\frac{p-1}{2}u^{(p+1)/2}\ubar^{(p-3)/2}\psibar^2\Bigr]\, dV \\
&=\frac{1}{2}((p+1)|u|^{p-1}\psi,\psi)
+\frac{p-1}{2}\Re (|u|^{p-3}u^2\psibar,\psi) \\
&=\frac{p+1}{2}(u^{p-1}\psi,\psi)+\frac{p-1}{2}\Re(u^{p-1}\psibar,\psi),
\endaligned
\label{3.5.37}
\een
the last identity by (\ref{3.5.8}).  Now
\ben
(u^{p-1}\psi,\psi)=(u^{p-1}\psi_0,\psi_0)+(u^{p-1}\psi_1,\psi_1),
\label{3.5.38}
\een
and
\ben
\aligned
\Re (u^{p-1}\psibar,\psi)&=\Re\int\limits_M u^{p-1}(\psi_0-i\psi_1)^2\, dV \\
&=(u^{p-1}\psi_0,\psi_0)-(u^{p-1}\psi_1,\psi_1),
\endaligned
\label{3.5.39}
\een
so
\ben
\aligned
\frac{d^2}{ds^2}\tJ(u+s\psi)\Bigr|_{s=0}
&=\frac{p+1}{2}(u^{p-1}\psi_0)+\frac{p-1}{2}(u^{p-1}\psi_0,\psi_0) \\
&\ \
+\frac{p+1}{2}(u^{p-1}\psi_1,\psi_1)-\frac{p-1}{2}(u^{p-1}\psi_1,\psi_1) \\
&=p(u^{p-1}\psi_0,\psi_0)+(u^{p-1}\psi_1,\psi_1) .
\endaligned
\label{3.5.40}
\een
Together with (\ref{3.5.28}), this gives
\ben
\aligned
\frac{d^2}{ds^2}E(u+s\psi)\Bigr|_{s=0}=\
&(-\Delta\psi_0-pu^{p-1}\psi_0,\psi_0) \\
&+(-\Delta \psi_1-u^{p-1}\psi_1,\psi_1).
\endaligned
\label{3.5.41}
\een
This, together with (\ref{3.5.23}) and (\ref{3.5.36}), yields
\ben
\aligned
\frac{d^2}{ds^2}E(w(s))\Bigr|_{s=0}=\
&((-\Delta-pu^{p-1})\psi_0,\psi_0)+\lambda(\psi_0,\psi_0) \\
&+((-\Delta-u^{p-1})\psi_1,\psi_1)+\lambda(\psi_1,\psi_1),
\endaligned
\label{3.5.42}
\een
when $w(s)$ is given by (\ref{3.5.7}).  In other words,
\ben
\frac{d^2}{ds^2}E(w(s))\Bigr|_{s=0}=(L_+\psi_0,\psi_0)+(L_-\psi_1,\psi_1),
\label{3.5.43}
\een
with $L_{\pm}:H^1(M)\rightarrow H^{-1}(M)$ given by
\ben
\aligned
L_+\psi_0&=(-\Delta+\lambda-p|u|^{p-1})\psi_0, \\
L_-\psi_1&=(-\Delta+\lambda-|u|^{p-1})\psi_1.
\endaligned
\label{3.5.44}
\een
The Friedrichs method defines $L_+$ and $L_-$ as self-adjoint operators
on $L^2(M)$, with domain $H^2(M)$.

\subsection{Second variation of $F_\lambda$}\label{s3.a}

In this setting, we take $p$ as in (\ref{E:p-range}), $J_p$ as in
(\ref{E:J-p}), and $F_\lambda$ as in (\ref{E:F-lambda}).
With $\beta\in (0,\infty)$, set
\ben
\wX=\{u\in H^1(M):J_p(u)=\beta\},\quad
I_\beta=\inf\, \{F_\lambda(u):u\in\wX\},
\label{3.A.1}
\een
and
\ben
\wY=\{u\in \wX:F_\lambda(u)=I_\beta\}.
\label{3.A.2}
\een
Conditions guaranteeing that $\wY$ is nonempty have been given in \S{\ref{s2}}.
Here we compute
\ben
\frac{d^2}{ds^2}F_\lambda(w(s)),
\label{3.A.3}
\een
where $w(s)$ is a smooth path in $\wX$ satisfying $w(0)=u\in\wY$.  To be
definite, take $u\in\wY$ (we can and will assume $u>0$), take
\ben
\psi\in T_u\wX=\{\psi\in H^1(M):\text{Re}\, (u^p,\psi)=0\},
\label{3.A.4}
\een
and set
\ben
w(s)=a\, \frac{u+s\psi}{\|u+s\psi\|_{L^{p+1}}},\quad
a=\|u\|_{L^{p+1}}=\beta^{1/(p+1)}.
\label{3.A.5}
\een
A calculation parallel to that done for $(d/ds)^2E(w(s))$ in \S{\ref{s3.5}},
which this time we leave to the reader, gives
\ben
\frac{1}{2}\, \frac{d^2}{ds^2} F_\lambda(w(s))\Bigr|_{s=0}
=(L_+\psi_0,\psi_0)+(L_-\psi_1,\psi_1),
\label{3.A.6}
\een
where $\psi=\psi_0+i\psi_1$, with $\psi_0,\psi_1$ real valued, and $L_{\pm}$
as in (\ref{3.5.44}), i.e.,
\ben
\aligned
L_+\psi_0&=(-\Delta+\lambda-p|u|^{p-1})\psi_0, \\
L_-\psi_1&=(-\Delta+\lambda-|u|^{p-1})\psi_1,
\endaligned
\label{3.A.7}
\een
provided a certain rescaling, described below, is performed.
In this case, the condition that $\psi\in H^1(M)$ belong to $T_u\wX$ becomes
\ben
(u^p,\psi_0)=0,
\label{3.A.8}
\een
with no further condition on $\psi_1$.  Contrast (\ref{3.A.8}) with
(\ref{3.5.10}).

We describe the rescaling of $u$ that yields (\ref{3.A.7}).  If $\wX$ is as in
(\ref{3.A.1}) and if $u\in \wY$ as in (\ref{3.A.2}) and is $\geq 0$, then there
exists $K\in\Rr$ such that such that
\ben
-\Delta u+\lambda u=Ku^p.
\label{3.A.9}
\een
Taking the inner product with $u$ yields
\ben
K=\beta^{-1} I_\beta.
\label{3.A.10}
\een
We can rescale, replacing $u$ by $\kappa u$,
to arrange that $K=1$ in (\ref{3.A.9}),
so, with a different $\beta$, the new $u$ minimizes $F_\lambda$ subject
to the constraint $J_p(u)=\beta$, and satisfies
\ben
-\Delta u+\lambda u-u^p=0.
\label{3.A.11}
\een
Cf.~Corollary \ref{c2.1.R}.  It is for this rescaled $u$ that
(\ref{3.A.6})--(\ref{3.A.7}) hold.

\subsection{Properties of $L_{\pm}$}\label{s3.6}

Throughout this subsection, $M$ will be a weakly homogeneous space.  We have
defined operators $L_+$ and $L_-$ in \S\S{\ref{s3.5}}--{\ref{s3.a}}, as
\ben
\aligned
L_+\psi_0&=(-\Delta+\lambda-p|u|^{p-1})\psi_0, \\
L_-\psi_1&=(-\Delta+\lambda-|u|^{p-1})\psi_1,
\endaligned
\label{3.6.0}
\een
arising when $u\in H^1(M)$ is either an energy minimizer of an
$F_\lambda$-monimizer, satisfying (\ref{3.A.11}).  We also assume $u>0$.
In light of decay results on $u$
established in \S{\ref{s2}}, these are self-adjoint operators on $L^2(M)$,
with domain $H^2(M)$.  In these respective cases, we have seen that,
with $w(s)$ respectively as in (\ref{3.5.7}) or (\ref{3.A.5}),
\ben
\frac{d^2}{ds^2}E(w(s))\Bigr|_{s=0}
=(L_+\psi_0,\psi_0)+(L_-\psi_1,\psi_1),
\label{3.6.1}
\een
\ben
\frac{d^2}{ds^2}F_\lambda(w(s))\Bigr|_{s=0}
=(L_+\psi_0,\psi_0)+(L_-\psi_1,\psi_1),
\label{3.6.2}
\een
with $\psi_j\in H^1(M)$ real valued, $\psi_1$ otherwise arbitrary, and
\ben
(u,\psi_0)&=0\quad \text{in case }\ (\ref{3.6.1}),
\label{3.6.3}
\een
\ben
(u^p,\psi_0)&=0\quad \text{in case }\ (\ref{3.6.2}).
\label{3.6.4}
\een
Since $u$ is a minimizer, we know that (\ref{3.6.1}) (resp., (\ref{3.6.2})) is
$\geq 0$ for all such paths $w(s)$.  We deduce that
\ben
(L_+\psi_0,\psi_0)\geq 0,
\label{3.6.5}
\een
for all real-valued $\psi_0\in H^1(M)$ satisfying (\ref{3.6.3}) when $u$ is
an energy minimizer, and for all $\psi_0$ satisfying (\ref{3.6.4}) when $u$
is an $F_\lambda$-minimizer (satisfying (\ref{3.A.11})).  Also, in both cases,
\ben
(L_-\psi_1,\psi_1)\geq 0,
\label{3.6.6}
\een
for all real valued $\psi_1\in H^1(M)$.  Since $L_+$ and $L_-$ are
reality preserving, these results extend readily to the case where $\psi_0$
and $\psi_1$ are allowed to be complex valued.

As we have seen, if $u\in H^1(M)$ minimizes $E(u)$, subject to the
constraint $\|u\|_{L^2}^2=\beta$, then there exists $\lambda\in\Rr$ such that
\ben
\Delta u-\lambda u+|u|^{p-1}u=0.
\label{3.6.7}
\een
(This is also the PDE satisfied by the rescaled $F_\lambda$-minimizer,
discussed in \S{\ref{s3.a}}.)  From (\ref{3.6.6}), we have the following
information about
$\lambda$.

\begin{proposition}\label{p3.6.1}
If $u$ is an energy minimizer satisfying (\ref{3.6.7}),
then
\ben
\Spec (-\Delta+\lambda)\subset [0,\infty).
\label{3.6.8}
\een
\end{proposition}

\begin{proof}
In fact, for all $\psi\in H^1(M)$,
\ben
\aligned
((-\Delta+\lambda)\psi,\psi)&=(L_-\psi,\psi)+(|u|^{p-1}\psi,\psi) \\
&\geq (|u|^{p-1}\psi, \psi) \\
&\geq 0.
\endaligned
\label{3.6.9}
\een
\end{proof}

By contrast with the positivity of $L_-$, note that
\ben
\aligned
(L_+ u,u)&=-(\Delta u-\lambda u+p|u|^{p-1}u,u) \\
&=-(p-1)\int_M |u|^{p+1}\, dV \\
&<0,
\endaligned
\label{3.6.9A}
\een
the second identity by (\ref{3.6.7}) if $u$ is an energy minimizer and by
(\ref{3.A.11}) if $u$ is an (appropriately rescaled) $F_\lambda$-minimizer.
This result, together with (\ref{3.6.5}), implies:

\begin{proposition}\label{p3.6.2}
If $u$ is either an energy minimizer or an
$F_\lambda$-minimizer, satisfying (\ref{3.A.11}), then $L_+$ has exactly one
negative eigenvalue.
\end{proposition}

Returning to $L_-$, we note that
\ben
L_-u=-(\Delta u-\lambda u+|u|^{p-1}u)=0,
\label{3.6.10}
\een
so
\ben
u\in\Cal{N}(L_-).
\label{3.6.11}
\een
We have the following more precise result.

\begin{proposition}\label{p3.6.3}
If $u$ is either an energy minimizer or an
$F_\lambda$-minimizer, satisfying (\ref{3.A.11}), then
\ben
\Cal{N}(L_-)=\Span (u).
\label{3.6.12}
\een
\end{proposition}

\begin{proof}
It suffices to show that any nonzero, real-valued element of
$\Cal{N}(L_-)$ must be either everywhere $>0$ or $<0$, since then no two
such can be orthogonal to each other.  Now, if $v\in \Cal{N}(L_-)$ is real
valued and $\|v\|_{L^2}=1$, then $v$ minimizes
\ben
\{\|\nabla v\|^2_{L^2}+\lambda\|v\|^2_{L^2}-(|u|^{p-1}v,v):\|v\|_{L^2}=1\}.
\label{3.6.13}
\een
Then $|v|$ is also minimizing, so $|v|\in\Cal{N}(L_-)$.  Then the Harnack
inequality implies $|v|>0$ on $M$, so indeed either $v>0$ or $v<0$ on $M$.
\end{proof}

We turn to some comments on $\Cal{N}(L_+)$.  In case $M$ has a $1$-parameter
group of isometries, generated by a vector field $\frX$ (known as a Killing
field), we get an element of $\Cal{N}(L_+)$ as follows.  Since Killing fields
commute with $\Delta$ we have (assuming $u>0$)
\ben
\Delta(\frX u)-\lambda(\frX u)+p|u|^{p-1}(\frX u)=0,
\label{3.6.14}
\een
hence
\ben
\frX u\in \Cal{N}(L_+),\quad \text{if $\frX$ is a Killing field on $M$},
\label{3.6.15}
\een
given estimates on $u$ assuring that $\frX u\in \Cal{D}(L_+)$.

It is useful to regard $L_+$ and $L_-$ as two operators in a continuum, defined
by
\ben
L_a\psi=-\Delta\psi+\lambda\psi-a|u|^{p-1}\psi,
\label{3.6.16}
\een
for $a\in\Rr$, particularly for $a\in [1,p]$.. Note that
\ben
L_1=L_-,\quad L_p=L_+.
\label{3.6.17}
\een
For each $a\in\Rr$, $L_a$ is self-adjoint on $L^2(M)$, with domain $\Cal{D}
(L_a)=H^2(M)$.  The following result extends Proposition \ref{p3.6.2}.

\begin{proposition}\label{p3.6.4}
Assume $u>0$ is either an energy minimizer or an $F_\lambda$-minimizer,
satisfying (\ref{3.A.11}).  Then, if $1<a\leq p$, $L_a$
has exactly one negative eigenvalue, and it has multiplicity one.
\end{proposition}

\begin{proof}
First note that
\ben
\aligned
(L_au,u)&=(L_1u+(1-a)u^{p-1}u,u) \\
&=-(a-1)\int_M u^{p+1}\, dV,
\endaligned
\label{3.6.18}
\een
which is $<0$ if $a>1$.  Next, for $\psi\in H^1(M)$,
\ben
\aligned
(L_a\psi,\psi)&=(L_p\psi+(p-a)u^{p-1}\psi,\psi) \\
&=(L_+\psi,\psi)+(p-a)\int_M u^{p-1}|\psi|^2\, dV,
\endaligned
\label{3.6.19}
\een
so, by (\ref{3.6.5}), if $u$ is an energy minimizer,
\ben
(u,\psi)=0\Longrightarrow (L_a\psi,\psi)\ge 0\ \text{ if }\ a\le p,
\label{3.6.20}
\een
while if $u$ is an $F_\lambda$-minimizer satisfying (\ref{3.A.11}),
\ben
(u^p,\psi)=0\Longrightarrow (L_a\psi,\psi)\ge 0\ \text{ if }\ a\le p.
\label{3.6.21}
\een
These results prove the proposition.
\end{proof}

The result (\ref{3.6.6}) implies $\Spec L_1\subset [0,\infty)$, and
Proposition \ref{p3.6.4} implies that
$\EsSp L_a\subset [0,\infty)$ for
$1<a\le p$.  We can say more about the essential spectrum.

\begin{proposition}\label{p3.6.5}
If $u$ is either an energy minimizer or an
$F_\lambda$-minimizer satisfying (\ref{3.A.11}), then, for all $a\in\Rr$,
\ben
\EsSp L_a=\EsSp (-\Delta+\lambda).
\label{3.6.22}
\een
\end{proposition}

\begin{proof}
Given $a\in\Rr$, pick $\mu>0$ so large that $L_a+\mu$ and
$-\Delta+(\lambda+\mu)$ are both invertible.  By Weyl's essential spectrum
theorem (\cite{RS-book}, p.~112) it suffices to note that
\ben
S_a=(L_a+\mu)^{-1}-(-\Delta+(\lambda+\mu))^{-1}\ \text{ is compact.}
\label{3.6.23}
\een
Recalling the formula (\ref{3.6.16}) for $L_a$, we have, by the resolvent
identity,
\ben
S_a=-a(L_a+\mu)^{-1}M_{|u|^{p-1}}(-\Delta+(\lambda+\mu))^{-1},
\label{3.6.24}
\een
whose compactness follows readily from the decay results given in
\S{\ref{s2.5}} and \S{\ref{s3.4}}, plus the Rellich theorem.
\end{proof}

For the next result, we assume the following:
\ben
\Spec (-\Delta+\lambda)\subset [\delta,\infty),\quad \delta>0.
\label{3.6.25}
\een
For $F_\lambda$-minimizers, this is equivalent to the hypothesis
(\ref{E:lam-del})--(\ref{E:delta-spec}).  For energy minimizers,
(\ref{3.6.25}) is slightly stronger than
(\ref{3.6.8}), and it can be expected to hold for almost all (if not all)
energy minimizers.

\begin{proposition}\label{p3.6.6}
Let $u$ be either an energy minimizer or an
$F_\lambda$-minimizer satisfying (\ref{3.A.11}), and assume (\ref{3.6.25})
holds.  Then
\ben
1<a<p\Longrightarrow \Cal{N}(L_a)=0.
\label{3.6.26}
\een
\end{proposition}

\begin{proof}
By Proposition \ref{p3.6.5} and (\ref{3.6.25}), for each $a$,
\ben
\text{Ess Spec}\, L_a\subset [\delta,\infty),\quad \delta>0.
\label{3.6.27}
\een
Suppose $a_0\in (1,p)$ and $\text{dim}\, \Cal{N}(L_{a_0})=m>0$.  The
Kato-Rellich theorem (\cite{RS-book}, p.~22) implies there exist analytic
functions $\lambda_j(a), 1\le j\le m$, for $a$ close to $a_0$, with
$\lambda_j(a_0)=0$, such that these are all the eigenvalues of $L_a$ near
$0$.  Also (\cite{RS-book}, p.~71)
there are corresponding eigenfunctions $\psi_{ja}$, analytic in $a$:
\ben
L_a\psi_{ja}=\lambda_j(a)\psi_{ja},\quad (\psi_{ja},\psi_{ka})\equiv
\delta_{jk},
\label{3.6.28}
\een
the orthonormality holding for $a$ real (and close to $a_0$).

Let us denote by $\psi^\#_a$ the (normalized) eigenfunction of $L_a$ given
by Proposition \ref{p3.6.4}.  We have
\ben
(\psi_{ja},\psi^\#_a)=0,
\label{3.6.29}
\een
for $j\in \{1,\dots,m\}$, real $a$ close to $a_0$.  Now apply $d/da$ to
(\ref{3.6.28}).  We get
\ben
-u^{p-1}\psi_{ja}+L_a\xi_{ja}=\lambda'_j(a)\psi_{ja}+\lambda_j(a)\xi_{ja},
\label{3.6.30}
\een
where
\ben
\xi_{ja}=\frac{d}{da}\psi_{ja}.
\label{3.6.31}
\een
The normalization in (\ref{3.6.28}) implies
\ben
(\psi_{ja},\xi_{ja})=0,
\label{3.6.32}
\een
so taking the inner product of (\ref{3.6.30}) with $\psi_{ja}$ gives
\ben
\lambda'_j(a)\|\psi_{ja}\|^2=-\int_M u^{p-1}|\psi_{ja}|^2\, dV,
\label{3.6.33}
\een
since
\ben
(L_a\xi_{ja},\psi_{ja})=(\xi_{ja},\lambda_j(a)\psi_{ja})=0.
\label{3.6.34}
\een
Hence, for $a$ close to 0,
\ben
\lambda'_j(a)<0,
\label{3.6.35}
\een
and if $\lambda_j(a_0)=0$, we get
\ben
\lambda_j(a)<0\ \text{ for }\ a_0<a<a_0+\ep,
\label{3.6.36}
\een
for some positive $\ep$.  This contradicts Proposition \ref{p3.6.4}, and
completes the proof.
\end{proof}

\subsection{Conditional orbital stability result}\label{s3.8}

We assume $M$ is a weakly homogeneous space, and $p$ satisfies
(\ref{E:subcrit}).  As in \S{\ref{s3.5}}, we fix $\beta>0$ and set
\ben
\aligned
X=\{u\in H^1(M):Q(u)&=\beta\},\quad \Cal{I}_\beta=\inf\,\{E(u):u\in X\}, \\
Y=\{u&\in X:E(u)=\Cal{I}_\beta\}.
\endaligned
\label{3.8.1}
\een
Under these hypotheses, the nonlinear Schr{\"o}dinger equation
\ben
iv_t+\Delta v+|v|^{p-1}v=0,\quad v(0)=v_0,
\label{3.8.2}
\een
is globally solvable, given $v_0\in H^1(M)$, via an argument given for $\Rr^n$
in \cite{SS}, \S{3.2.2}.  Conservation of mass and energy imply that $X$ and
$Y$ are invariant under the solution operator to (\ref{3.8.2}).
We investigate the following question concerning orbital stability.
Assume
\ben
v_0\in X
\label{3.8.2A}
\een
is close to $Y$ (distance measured in $H^1(M)$-norm).  We then ask whether
the solution $v(t)$ to (\ref{3.8.2}) can be shown to be close to $Y$, for
all $t\in\Rr$.  Since energy is conserved for solutions to (\ref{3.8.2}):
\ben
E(v(t))\equiv E(v_0),
\label{3.8.3}
\een
a positive result would follow if one could show that if $u\in X$ and $E(u)$
is close to $\Cal{I}_\beta$, then $u$ is close to $Y$.

We establish such a result, under the following two assumptions.
The first is an essential uniqueness hypothesis:
\ben
\aligned
&\text{If $u_1, u_2$ are positive functions in $Y$, there is an isometry} \\
&\varphi:M\rightarrow M
\ \text{ such that $u_2=u_1\circ\varphi$.}
\endaligned
\label{3.8.4}
\een
Recall that if $u\in Y$, there exists $\lambda\in\Rr$ such that
\ben
-\Delta u+\lambda u-|u|^{p-1}u=0.
\label{3.8.5}
\een
The hypothesis (\ref{3.8.4}) implies that (\ref{3.8.5})
holds with the same $\lambda$ for all $u\in Y$.
Our second hypothesis is that (\ref{3.6.25}) hold, i.e.,
\ben
\Spec (-\Delta+\lambda)\subset [\delta,\infty),\quad \text{for some }\
\delta>0,
\label{3.8.6}
\een
which, recall, is slightly stronger than (\ref{3.6.8}).

To state our first result, let $\Cal{G}$ denote the group of operators
on functions on $M$ of the form
\ben
u(x)\mapsto e^{i\theta}u(\varphi(x)),\quad \theta\in\Rr,\ \ \varphi:M
\rightarrow M\ \text{isometry}.
\label{3.8.7}
\een
Thus $\Cal{G}$ acts as a group of isometries on $L^2(M)$ and on $H^1(M)$,
preserving $X$ and $Y$.  The following is immediate.

\begin{proposition}\label{p3.8.1}
Under the hypothesis (3.8.4), $\Cal{G}$ acts
transitively on $Y$, and $Y$ is a smooth, finite dimensional submanifold of
$X$.
\end{proposition}

It is this result that puts the ``orbital'' in ``orbital stability.''
In case $M=\Rr^n$, $Y$ (shown in \S{\ref{sa3}} to be nonempty)
is diffeomorphic to $\Rr^n\times S^1$ (granted hypothesis (\ref{3.8.4}),
also demonstrated for $\Rr^n$ in \S{\ref{sa3}}).
In other cases, the group of
isometries of $M$ might be discrete and $Y$ would be 1-dimensional.

To proceed, for $\ep>0$, set
\ben
\Cal{O}_\ep =\{u\in X:\text{dist}_{H^1}(u,Y)\le\ep\}.
\label{3.8.8}
\een
Then $\Cal{O}_\ep$ is invariant under the action of $\Cal{G}$.  By Proposition
\ref{p3.8.1}, if $\ep$ is sufficiently small, given $u\in Y$, $\Cal{O}_\ep$ is
swept out by the $\Cal{G}$-action on a codimension-$m$ space $\Sigma$, normal
to $Y$ at $u$ (with $m=\text{dim}\, Y$).

The following is an orbital stability result.

\begin{proposition}\label{p3.8.2}
Assume hypotheses (\ref{3.8.4}) and (\ref{3.8.6}).
For $\ep>0$ sufficiently small, the following holds.
If $v_\nu\in\Cal{O}_\ep$
and $E(v_\nu)\rightarrow \Cal{I}_\beta$, then
\ben
\dist_{H^1}(v_\nu,Y)\rightarrow 0.
\label{3.8.9}
\een
\end{proposition}

Note that we can take $\tilde{v}_\nu\in\Sigma$ such that $E(\tilde{v}_\nu)=
E(v_\nu)$, and dist$(\tilde{v}_\nu,Y)=\text{dist}(v_\nu,Y)$,
so without loss of generality we can assume $v_\nu\in \Sigma$.  We will
parametrize an appropriate space $\Sigma$ by a neighborhood of $0$ in
an $\Rr$-linear subspace $V$ of $T_uX$, of codimension $m$, as follows.
We set
\ben
V=\{\psi\in T_uX:\psi\perp T_u Y\}.
\label{3.8.10}
\een
Recall the characterization of $T_uX$ in (\ref{3.5.6}),
supplemented by (\ref{3.5.9})--(\ref{3.5.10}).
$V$ is an $\Rr$-linear subspace of $H^1(M)$, of codimension $m+1$, a Hilbert
space with the $H^1$-norm.

To proceed, we define a function $F$ on a neighborhood of $0\in V$ by
\ben
F(\psi)=E\Bigl(a\, \frac{u+\psi}{\|u+\psi\|}\Bigr).
\label{3.8.11}
\een
We have
\ben
F(0)=E(u)=\Cal{I}_\beta,\quad DF(0)=0,
\label{3.8.12}
\een
and calculations of \S{3.5} give
\ben
D^2F(0)(\psi,\psi)=(L_+\psi_0,\psi_0)+(L_-\psi_1,\psi_1).
\label{3.8.13}
\een

In light of this, Proposition \ref{p3.8.2} is a consequence of the following.

\begin{proposition}\label{p3.8.3}
Let $V$ be a real Hilbert space, $\Cal{B}_r$ a
ball of radius $r$
centered at $0\in V$, and $F:\Cal{B}_r\rightarrow \Rr$ a $C^2$ function
satisfying the following conditions:
\ben
F(0)=\Cal{I}_\beta,\quad \psi\in\Cal{B}_r\setminus 0\Rightarrow
F(\psi)>\Cal{I}_\beta
\label{3.8.14}
\een
(so $DF(0)=0$).  Also assume there exists $C>0$ and an orthogonal projection
$P:V\rightarrow V$, with range of finite codimension, such that,
for $\psi\in V$,
\ben
D^2F(0)(\psi,\psi)\ge C\|P\psi\|^2_V.
\label{3.8.15}
\een
Then, if $\rho\in (0,r)$ is small enough,
\ben
\psi_\nu\in\Cal{B_\rho},\ F(\psi_\nu)\rightarrow \Cal{I}_\beta\Longrightarrow
\|\psi_\nu\|_V\rightarrow 0.
\label{3.8.16}
\een
\end{proposition}

\begin{proof}
Taylor's formula with remainder gives
\ben
F(\psi)=\Cal{I}_\beta+\frac{1}{2} D^2F(0)(\psi,\psi)+R(\psi),
\label{3.8.17}
\een
with
\ben
R(\psi)=\int_0^1 [D^2F(t\psi)-D^2F(0)](\psi,\psi)\, (1-t)\, dt
=o(\|\psi\|^2_V).
\label{3.8.18}
\een
Hence, if $\psi\in\Cal{B}_\rho$ and $\rho$ is small enough,
\ben
\aligned
F(\psi)&\geq\Cal{I}_\beta+\frac{C}{2}\|P\psi\|^2_V-o(\|\psi\|^2_V) \\
&\geq \Cal{I}_\beta+\frac{C}{4}\|P\psi\|^2_V-o(\|P^\perp \psi\|^2_V),
\endaligned
\label{3.8.19}
\een
where $P^\perp =I-P$ has finite dimensional range.  Hence the hypothesis
(\ref{3.8.16}) on $\psi_\nu$ implies
\ben
\|P\psi_\nu\|_V\longrightarrow 0.
\label{3.8.20}
\een
We need to show that $P^\perp \psi_\nu\rightarrow 0$ in $W=\text{Range}\,
P^\perp\subset V$.  The sequence $(P^\perp \psi_\nu)$ is a bounded sequence
in $W$, so $(\psi_\nu)$ has a subsequence (which we continue to denote
$(\psi_\nu)$) such that $P^\perp \psi_\nu\rightarrow \tilde{\psi}$.
Hence $\psi_\nu\rightarrow \tilde{\psi}$.
Now $F(\psi_\nu)\rightarrow \Cal{I}_\beta$
implies $F(\tilde{\psi})=\Cal{I}_\beta$.  The hypothesis (\ref{3.8.14}) then
gives $\tilde{\psi}=0$, and completes the proof.
\end{proof}

$\text{}$\newline
{\bf Remark.} In the setting of Proposition \ref{p3.8.3}, the range of
$P^\perp$ is the orthogonal complement of $T_uY$ in
\ben
\{\psi=\psi_0+i\psi_1\in T_uX:(L_+\psi_0,\psi_0)+(L_-\psi_1,\psi_1)=0\},
\label{3.8.21}
\een
which is a linear space, by (\ref{3.6.5})--(\ref{3.6.6}), and is finite
dimensional, given (\ref{3.8.6}), by (\ref{3.6.27}).

$\text{}$ \newline
{\bf Remark.} One setting where Proposition \ref{p3.8.2} applies is that of
Euclidean space, $M=\Rr^n$.  In this case, the uniqueness hypothesis
(\ref{3.8.4}) and the spectral hypothesis (\ref{3.8.6}) follow from Proposition
\ref{pa.3.3}.  In this case, orbital stability was established in
\cite{Wei86}.  Further applications of Proposition \ref{p3.8.2} are being
pursued in \cite{CMMT}.

\section{Exploration of symmetrization techniques}\label{s4}

As mentioned in the Introduction, works of \cite{Str} and \cite{BPL}
used a symmetrization technique to construct ground states on Euclidean
space, namely $F_\lambda$-minimizers in \cite{Str} and minimizers of
$\|\nabla u\|^2_{L^2}$ subject to the constraint (\ref{1.0.12}) in
\cite{BPL}.  Here we explore other applications of such a symmetrization
technique.

Behind this approach is a key rearrangement lemma.  We state this result
and say a little about how it has been
proved in \S{\ref{s4.1}}, and then proceed to applications in
\S\S{\ref{sa4}--\ref{sa5}}.

We use the technique to produce $F_\lambda$-minimizers
on hyperbolic space in \S{\ref{sa4}}.
Here, we make use of arguments from \cite{ChMa-hyp}, but with simplifications,
which allow us to completely avoid appeal to concentration-compactness
arguments.  We obtain a unified treatment of $F_\lambda$-minimizers
on hyperbolic space and on Euclidean space.

In \S{\ref{sa4x}}, we apply the symmetrization technique to the task
of maximizing the Weinstein functional.  For Euclidean space, this
provides a short and direct proof of existence of such maximizers.
The dilation structure of Euclidean space plays a crucial role,
and we note myriad examples of Riemannian manifolds for which the
Weinstein functional does not have a maximum.

In \S{\ref{sa5}}, we discuss the symmetrization approach to the existence
of energy minimizers.  In this case, this approach seems to stop short of
actually establishing the existence of such minimizers, though we do
obtain some interesting information.

\subsection{The rearrangement lemma}\label{s4.1}

Here is the key rearrangement lemma.

\begin{lemma}\label{la4.1}
If $M=\Rr^n$ or $\Cal{H}^n$,
replacing $u\in H^1(M)$ by its radial decreasing rearrangement
does not increase $\|\nabla u\|_{L^2}$.
\end{lemma}

Proofs have been given in \cite{Str} and in \cite{BPL} when $M=\Rr^n$.
The result for $M=\Cal{H}^n$ was established in \cite{ChMa-hyp}.  The proof
requires two nontrivial ingredients.  One is heat kernel monotonicity.
This has been established, on all rank-one symmetric spaces, using exact
formulas for the heat kernel.

The other ingredient is an integral rearrangement inequality.
This rearrangement inequality holds for $M=\Rr^n$ or $\Cal{H}^n$.
In the former case, it is a consequence of a general rearrangement
inequality of \cite{BLL}, in the latter case, \cite{Bek} produced the
extension of such a rearrangement inequality to hyperbolic space.
The proof of the rearrangement inequality
requires that $M$ be a rank-one symmetric space, and further that $M$ 
possess reflection symmetry, across a totally geodesic hypersurface.
Such reflection symmetry fails for the other noncompact
rank-one symmetric spaces.

\subsection{Symmetrization approach to $F_\lambda$-minimizers}
\label{sa4}

Here we take $M$ to be either $n$-dimensional hyperbolic space $\Cal{H}^n$
or Euclidean space $\Rr^n$, $n\geq 2$.
As before, we define $\delta_0$ to be the smallest number satisfying
\ben
\text{Spec}(-\Delta)\subset [\delta_0,\infty).
\label{1}
\een
If $M=\Rr^n$, we have $\delta_0=0$.  If $M=\Cal{H}^n$,
we have $\delta_0=(n-1)^2/4$.
As in \S{\ref{s2}}, we assume
\ben
\lambda>-\delta_0,\quad p+1\in\Bigl(2,\frac{2n}{n-2}\Bigr),\quad
\beta\in (0,\infty).
\label{2}
\een
We aim to minimize
\ben
F_\lambda(u)=\|\nabla u\|^2_{L^2}+\lambda\|u\|^2_{L^2},
\label{3}
\een
subject to the constraint
\ben
J_p(u)=\int\limits_M |u|^{p+1}\, dV=\beta.
\label{4}
\een
This was accomplished in \S{\ref{s2}}, in the more general setting of weakly
homogeneous spaces.  Here we use the symmetrization method.

We turn to the task of finding the desired minimizer.
Note that (\ref{1})--(\ref{2}) imply
\ben
F_\lambda(u)\approx \|u\|^2_{H^1(M)},
\label{5}
\een
and in this setting we have the Sobolev embedding result
\ben
H^1(M)\subset L^q(M),\quad \forall\, q\in \Bigl[2,\frac{2n}{n-2}\Bigr],
\label{6}
\een
if $n\ge 3$, $\forall\, q\in [2,\infty)$ if $n=2$.
The results (\ref{5})--(\ref{6}) imply
\ben
\|u\|^2_{L^{p+1}}\le CF_\lambda(u),
\label{14}
\een
so
\ben
I_\beta=\inf\, \{F_\lambda(u):J_p(u)=\beta\}>0.
\label{15}
\een
Let $u_\nu\in H^1(M)$ satisfy
\ben
J_p(u_\nu)=\beta,\quad F_\lambda(u_\nu)\le I_\beta+\frac{1}{\nu}.
\label{16}
\een
Passing to a subsequence, which we continue to denote $(u_\nu)$, we have
\ben
u_\nu\longrightarrow u\in H^1(M),\quad \text{converging weakly.}
\label{17}
\een
Rellich's theorem gives
\ben
H^1(\Omega)\hookrightarrow L^{p+1}(\Omega)\ \text{ compact,}
\label{18}
\een
for all smoothly bounded $\Omega\subset M$, as long as $p+1$ satisfies
(\ref{2}), so
\ben
u_\nu\longrightarrow u,\ \text{ in $L^{p+1}(\Omega)$ norm,}
\label{19}
\een
for all such $\Omega\subset M$.
Fix a base point $o\in M$, and replace $u_\nu$ by its radial decreasing
rearrangement.  By Lemma \ref{la4.1}, this replacement does not
increase $\|\nabla u_\nu\|_{L^2}$.  On the other hand,
such a replacement clearly leaves $\|u_\nu\|_{L^2}$ fixed, hence lowers
$F_\lambda(u_\nu)$.  It also leaves $J_p(u_\nu)$ fixed.
Thus we can assume $u_\nu(x)\ge 0$, that it is rotationally symmetric about
$o\in M$, that it is monotone in $\text{dist}(x,o)$, and that (\ref{16}),
(\ref{17}) and (\ref{19}) hold.  We need to show that
\ben
J_p(u)=\beta,\quad \text{i.e., }\ \|u\|_{L^{p+1}}=\beta^{1/(p+1)}.
\label{20}
\een
Clearly $J_p(u)\le \beta$ and $F_\lambda(u)\le I_\beta$.  Given (\ref{20}),
it would follow from (\ref{15}) and (\ref{17}) that
\ben
F_\lambda(u)=I_\beta,
\label{21}
\een
and also that $H^1$-norm convergence holds in (\ref{17}).

To demonstrate (\ref{20}), let us set
$\|u\|_{H^1}=\|\nabla u\|_{L^2(M)}+\|u\|_{L^2(M)}$,
$H^1_r(M)=\text{ set of radially symmetric functions in } H^1(M)$,
$M_R=M\setminus B_R(o)$, and
\ben
J_Rv=v\Bigr|_{M_R}.
\label{22}
\een
We assert the following.

\begin{lemma}\label{la4.2}
Given $q>2$, we have
\ben
\lim\limits_{R\rightarrow\infty}\, \|J_R\|_{\Cal{L}(H^1_r,L^q)}=0.
\label{23}
\een
\end{lemma}

Given this lemma, we have for the radial sequence $(u_\nu)$ satisfying
(\ref{16}) that, for each $\ep>0$, there exists $R<\infty$ such that
\ben
\int\limits_{M_R} |u_\nu|^{p+1}\, dV\le\ep,\quad
\forall\, \nu,
\label{24}
\een
and then (\ref{20}) follows from (\ref{19}).

It remains to prove Lemmma \ref{la4.2}.  If we show that
\ben
\lim\limits_{R\rightarrow\infty}\, \|J_R\|_{\Cal{L}(H^1_r,L^\infty)}=0,
\label{25}
\een
then, since for $q>2$
\ben
\aligned
\int\limits_{M_R}|v|^q\, dV
&\le \|v\|^{q-2}_{L^\infty(M_R)} \int\limits_{M_R}|v|^2\, dV \\
&\le \|J_Rv\|^{q-2}_{L^\infty} \|v\|^2_{H^1},
\endaligned
\label{26}
\een
we have (\ref{23}).

It remains to prove (\ref{25}).  Here is one approach.  We can replace $R$ by $R+1$.
Take $\chi_R\in\text{Lip}(M),\ \chi_R(x)=0$ for $x\in B_R(o),\ \chi_R(x)=
\text{dist}(x,B_R(o))$ for $x\in B_{R+1}(o)$, $\chi_R(x)=1$ for $x\in M_{R+1}$.
Then, for $v\in H^1_r(M)$, we have
\ben
\chi_Rv\in H^1_{0,r}(M_R)=H^1_0(M_R)\cap H^1_r(M),
\een
and
\ben
\|\nabla(\chi_R v)\|_{L^2}\le \|v\|_{H^1}.
\label{27}
\een
Hence, in all cases except $M=\Rr^2$, (\ref{25}) is a consequence of the
following.

\begin{lemma}\label{la4.3}
Except for $M=\Rr^2$, we have
\ben
\|v\|_{L^\infty}\le \eta(R)\|\nabla v\|_{L^2},\quad \forall v\in
H^1_{0,r}(M_R),
\label{28}
\een
with
\ben
\lim\limits_{R\rightarrow\infty}\, \eta(R)=0.
\label{29}
\een
\end{lemma}
\begin{proof} Take $v\in H^1_{0,r}(M_R)$.  Slightly abusing notation, we write
$v(x)=v(r)$.  Then
\ben
\|\nabla v\|^2_{L^2}=\int_R^\infty |v'(r)|^2 A(r)\, dr,
\label{30}
\een
where
\ben
A(r)=(n-1)\text{-dimensional area of } \{x\in M:\text{dist}(x,o)=r\}.
\label{31}
\een
Now
\ben
\aligned
\|v\|_{L^\infty}&\le \int_R^\infty |v'(r)|\, dr \\
&=\int_r^\infty |v'(r)|A(r)^{1/2}A(r)^{-1/2}\, dr \\
&\le \eta(R)\|\nabla v\|_{L^2},
\endaligned
\label{32}
\een
by Cauchy's inequality, where
\ben
\eta(R)=\Bigl(\int_R^\infty \frac{dr}{A(r)}\Bigr)^{1/2}.
\label{33}
\een
This gives (\ref{28}), except when $M=\Rr^2$.
In fact, $A(r)=A_n r^{n-1}$ when $M=\Rr^n$, and it blows up exponentially
when $M$ is $\Cal{H}^n$.
\end{proof}

Finally, the case $M=\Rr^2$ of (\ref{25}) follows from the next result,
given in \cite{BPL} as Radial Lemma A.II, which in turn follows \cite{Str}.

\begin{lemma}\label{la4.4}
If $M=\Rr^n,\ n\ge 2$, then, for $R\ge 1$,
\ben
\sup\limits_{|x|=R}\, |v(x)|\le C_n R^{-(n-1)/2} \|v\|_{H^1},\quad
\forall\, v\in H^1_{0,r}(M_1).
\label{34}
\een
\end{lemma}

\subsection{Symmetrization approach to Weinstein functional maximization}
\label{sa4x}

Complementing \S{\ref{sa4}}, we note how the symmetrization procedure
allows for a simplified proof of the existence of a maximum for the Weinstein
functional $W(u)$ in (\ref{1.3A.3}),
in the Euclidean space setting, $\Rr^n$ (for $n\geq 2$).  Recall,
\ben
W(u)=\frac{\|u\|_{L^{p+1}}^{p+1}}{\|u\|_{L^2}^\alpha \|\nabla u\|^\beta_{L^2}},
\label{A.4X.0}
\een
with $\alpha=2-(n-2)(p-1)/2,\ \beta=n(p-1)/2$.
We keep the requirement (\ref{2}) on $p$.
The Gagliardo-Nirenberg estimate implies $W(u)$ is bounded from above.
Denote its supremum by $W_{\max}$.

Now, if $u_\nu\in H^1(\Rr^n)$ and $W(u_\nu)\rightarrow W_{\max}$, then
$W(u^*_\nu)\ge W(u_\nu)$ if $u^*_\nu$ is the radial decreasing rearrangement
of $u_\nu$, so we need only maximize $W(u)$ over $H^1_r(\Rr^n)$.  For the
next step, we follow the standard argument and use the fact that $W(u)$
is invariant under $u\mapsto au$ and $u(x)\mapsto u(bx)$ to impose the
normalization
\ben
\|u_\nu\|_{L^2}=1,\quad \|\nabla u_\nu\|_{L^2}=1,
\label{A.4X.1}
\een
so
\ben
\|u_\nu\|_{L^{p+1}}\rightarrow W^{1/(p+1)}_{\max}.
\label{A.4X.2}
\een
If we pass to a subsequence such that $u_\nu\rightarrow u$ $\text{weak}^*$ in
$H^1(\Rr^n)$, results from \S{A.4} yield $u_\nu\rightarrow u$ in norm in
$L^{p+1}(\Rr^n)$.  Also $\|u\|_{L^2}\le 1$ and $\|\nabla u\|_{L^2}\le 1$, so
\ben
W(u)\ge W_{\max}.
\label{A.4X.3}
\een
This requires $W(u)=W_{\max}$ (hence $\|u\|_{L^2}=\|\nabla u\|_{L^2}=1$,
and therefore $u_\nu\rightarrow u$ in norm in $H^1(\Rr^n)$.)
We have the desired maximizer.
A computation of
\ben
\frac{d}{d\tau} W(u+\tau v)\Bigr|_{\tau=0}
=\frac{(N(u),v)}{\|u\|^{2\alpha}_{L^2}\|\nabla u\|^{2\beta}_{L^2}}
\label{A.4X.4}
\een
shows that such a maximizer $u$ solves the equation
\ben
\Delta u-\lambda u+Ku^p=0,
\label{A.4X.5}
\een
with
\ben
\lambda=\frac{\alpha}{\beta}\,
\frac{\|\nabla u\|^2_{L^2}}{\|u\|^2_{L^2}},\quad
K=\frac{p+1}{\beta}\,
\frac{\|\nabla u\|^2_{L^2}}{\|u\|^{p+1}_{L^{p+1}}},
\label{A.4X.6}
\een
hence, with the normalization imposed above,
\ben
\lambda=\frac{\alpha}{\beta},\quad K=\frac{p+1}{\beta W_{\max}}.
\label{A.4X.7}
\een

By contrast, note the following non-existence result.

\begin{proposition}\label{pa.4x.1}
Let $\Omega\subset\Rr^n$ be a nonempty open set
such that $\Rr^n\setminus\Omega$ has positive capacity.  Then
\ben
\{W(u):u\in H^1_0(\Omega)\}
\label{A.4X.8}
\een
does not achieve a maximum.
\end{proposition}

\begin{proof}
Denote the supremum of (\ref{A.4X.8}) by $W^\Omega_{\max}$.  Then
$W^\Omega_{\max}\le W^{\Rr^n}_{\max}$, since the supremum of (\ref{A.4X.8})
is over a subset of $H^1(\Rr^n)$.  On the other hand, taking a maximizer
of $W(u)$ over $H^1(\Rr^n)$, dilating it, to be highly concentrated near a
point $p\in\Omega$, and using a cutoff, we see that $W^\Omega_{\max}\ge
W^{\Rr^n}_{\max}$, so in fact $W^\Omega_{\max}=W^{\Rr^n}_{\max}$.
If $v\in H^1_0(\Omega)$ and $W(v)=W^\Omega_{\max}$, we can replace $v$
by $|v|$ and arrange $v\ge 0$.  Then extending $v$ by $0$ on $\Rr^n\setminus
\Omega$ would yield $u\in H^1(\Rr^n)$ such that $W(u)=W^{\Rr^n}_{\max}$.
Then $u\ge 0$ would solve (\ref{A.4X.5}).
By Harnack's inequality, that would
force $u>0$ on $\Rr^n$, yielding a contradiction.
\end{proof}

Turning to the setting of hyperbolic space $\Cal{H}^n$, we do not have
dilations, and cannot achieve the normalization (\ref{A.4X.1}), when taking
$u_\nu\in H^1(\Cal{H}^n)$ such that $W(u_\nu)\rightarrow W_{\max}$.
We can arrange
that
\ben
\|\nabla u_\nu\|_{L^2}=1,
\label{A.4X.9}
\een
which implies $\|u_\nu\|_{L^2}$ and $\|u_\nu\|_{L^{p+1}}$ are bounded.
Again, $u_\nu$ can be arranged to be radial (and decreasing).
Take a subsequence $u_\nu\rightarrow u$ $\text{weak}^*$ in $H^1(M)$.
From here, there are two scenarios to consider.  After perhaps passing
to a further subsequence, either
\ben
\|u_\nu\|_{L^2}\longrightarrow A>0,\quad \text{(Case I)},
\label{A.4X.10}
\een
or
\ben
\|u_\nu\|_{L^2}\longrightarrow 0,\qquad \text{(Case II).}
\label{A.4X.11}
\een

In Case I, we have
\ben
\|u_\nu\|^{p+1}_{L^{p+1}}\longrightarrow A^\alpha W_{\max},
\label{A.4X.12}
\een
and $u_\nu\rightarrow u$ in $L^{p+1}$-norm, so $\|u\|^{p+1}_{L^{p+1}}=
A^\alpha W_{\max}$.  Also, $\|u\|_{L^2}\le A$ and $\|\nabla u\|_{L^2}\le 1$,
so
\ben
W(u)\geq \frac{A^\alpha W_{\max}}{A^\alpha}=W_{\max}.
\label{A.4X.13}
\een
Hence
\ben
W(u)=W_{\max},\quad \|u\|_{L^2}=A,\quad \|\nabla u\|_{L^2}=1,
\label{A.4X.14}
\een
so $u_\nu\rightarrow u$ in $H^1$-norm, and we have a Weinstein functional
maximizer.  It solves the PDE (\ref{A.4X.5}), with $\lambda$ and $K$ given by
(\ref{A.4X.6}), i.e., in this case,
\ben
\lambda=\frac{\alpha}{\beta}\, \frac{1}{A^2},\quad
K=\frac{p+1}{\beta A^\alpha W_{\max}}.
\label{A.4X.15}
\een

In Case II, we have
\ben
\|u_\nu\|_{L^{p+1}}\longrightarrow 0,\quad \text{and }\ u=0.
\label{A.4X.16}
\een
In such a case $(u_\nu)$ does not converge to a $W$-maximizer.  Note that,
if there is a $W$-maximizer $u$, there must be a sequence $u_\nu$ satisfying
(\ref{A.4X.9}), $W(u_\nu)\rightarrow W_{\max}$, and (\ref{A.4X.10})
(just take $u_\nu\equiv u$).  Thus we pose the following
question. For $M=\Cal{H}^n$, is there a sequence $u_\nu\in H^1_r(M)$
satisfying (\ref{A.4X.9}) and $W(u_\nu)\rightarrow W_{\max}$, such that
(\ref{A.4X.10}) holds, or must (\ref{A.4X.11}) hold?

In connection with this, we note that part of the proof of
Proposition \ref{pa.4x.1} extends to give
\ben
W^\Omega_{\max}\ge W^{\Rr^n}_{\max},
\label{A.4X.17}
\een
for any Riemannian manifold with boundary $\Omega$, where $W^\Omega_{\max}$
is the supremum of (\ref{A.4X.8}).  It is tempting to conjecture that
\ben
W(u)<W^{\Rr^n}_{\max},\quad \forall\, u\in H^1(\Cal{H}^n),
\label{A.4X.18}
\een
and hence $W_{\max}$ is not achieved in $H^1(\Cal{H}^n)$.

On the other hand, there are Riemannian manifolds with boundary $\Omega$
such that
\ben
W^\Omega_{\max}>W^{\Rr^n}_{\max}.
\label{A.4X.19}
\een
One can, for example, let $M$ be a compact, connected Riemannian manifold
without boundary and let $\Omega=M\setminus B$, where $B\subset M$ is a small
ball.  It would be interesting to know whether $W^\Omega_{\max}$ can be
achieved in such cases.

\subsection{Symmetrization approach to energy minimizers}\label{sa5}

We retain the setting of \S{\ref{sa4}}, and assume $M$ is a (noncompact)
$n$-dimensional, rank-one symmetric space with reflection symmetry, i.e.,
$M=\Rr^n$ or $\Cal{H}^n$ ($n\geq 2$).
We require on $p$ the more stringent condition
\ben
p\in \Bigl(1,1+\frac{4}{n}\Bigr).
\label{C.0}
\een
We fix $\beta>0$,
and pick $u_\nu\in H^1(M)$ such that
\ben
Q(u_\nu)=\beta,\quad E(u_\nu)\leq \Cal{I}_\beta+\frac{1}{\nu},
\label{C.0A}
\een
with
\ben
\aligned
Q(u)&=\|u\|^2_{L^2},\quad
E(u)=\frac{1}{2}\|\nabla u\|^2_{L^2}-\frac{1}{p+1}\|u\|^{p+1}_{L^{p+1}}, \\
\Cal{I}_\beta&=\inf\, \{E(u):u\in H^1(M),\, Q(u)=\beta\}.
\endaligned
\een
As seen in \S{\ref{s3}},
this leads to bounds
\ben
\|u_\nu\|_{H^1},\ \|u_\nu\|_{L^{p+1}}\le K<\infty.
\label{C.1}
\een
To proceed, fix a base point $o\in M$.
We make use of Lemma \ref{la4.1}, which implies that, if $M=\Rr^n$ or $\Cal{H}^n$,
replacing $u_\nu$ by its radial decreasing rearrangement does not increase
$\|\nabla u_\nu\|_{L^2}$.  Also, such a replacement leaves $Q(u_\nu)$
invariant.  Thus we can assume our
minimizing sequence $(u_\nu)$ consists of such radial decreasing functions.
Passing to a subsequence, we have
\ben
u_\nu\longrightarrow u,\quad \text{weak}^*\ \text{ in }\ H^1(M),
\label{C.2}
\een
hence $\text{weak}^*$ in $L^2(M)$ and in $L^{p+1}(M)$.  The limit $u$
is radial and decreasing.  The next result provides valuable information
about $E(u)$.

\begin{proposition}\label{pc.2}
For such a sequence $(u_\nu)$, we have
\ben
u_\nu\longrightarrow u\ \text{ in }\ L^{p+1}(M)\text{-norm.}
\label{C.3}
\een
\end{proposition}
\begin{proof} As long as $p+1<2n/(n-2)$ (which is a weaker requirement than
(\ref{C.0})), Rellich's theorem gives
\ben
H^1(B_R(o))\hookrightarrow L^{p+1}(B_R(o))\ \text{ compact,}
\label{C.4}
\een
for each $R<\infty$, where $B_R(o)$ is the ball of radius $R$ centered at
$o$.  Hence we have
\ben
u_\nu\longrightarrow u\ \text{ in }\ L^{p+1}(B_R(o))\text{-norm},\quad
\forall\, R<\infty.
\label{C.5}
\een
To proceed, denote by $H^1_r(M)$ the space of radially symmetric functions
in $H^1(M)$.  Set
\ben
M_R=M\setminus B_R(o),\quad J_Rv=v\Bigr|_{M_R}.
\label{C.6}
\een
The following complement to (\ref{C.5})
follows from Lemma \ref{la4.2}.
Namely, given $q>2$, we have, for $v\in H^1_r(M)$,
\ben
\|J_Rv\|_{L^q}\le \delta_q(R)\|v\|_{H^1_r},\quad \delta_q(R)\rightarrow 0\
\text{ as }\ R\rightarrow\infty.
\label{C.7}
\een
Consequently, we have for the radial sequence $(u_\nu)$ that, for
each $\ep>0$, there exists $R<\infty$ such that
\ben
\int\limits_{M_R}|u_\nu|^{p+1}\, dV\le\ep,\quad \forall\,\nu,
\label{C.8}
\een
and then (\ref{C.3}) follows from (\ref{C.5}), proving Proposition \ref{pc.2}.
\end{proof}

$\text{}$

From (\ref{C.2}) we have
$\|\nabla u\|_{L^2}\le \liminf \|\nabla u_\nu\|_{L^2}$, and this together with
(\ref{C.3}) gives
\ben
E(u)\le \liminf\limits_{\nu\rightarrow\infty}\, E(u_\nu)=\Cal{I}_\beta.
\label{C.9}
\een
We'd like to know that
\ben
Q(u)=\beta.
\label{C.10}
\een
However, Lemma \ref{la4.2} requires $q>2$, so it is not clear how to establish
(\ref{C.10}) directly.

Consequently, even in the current setting, $M=\Rr^n$ or $\Cal{H}^n$, the
energy-minimizer existence result seems to need the concentration-compactness
argument given in \S{\ref{s3}} (which at present requires negative energy).

\appendix

\section{ }
As mentioned in the Introduction, we have four appendices.
Appendix \ref{sa1} presents the concentration-compactness method
of P.-L.~Lions, in the setting of a class of measured metric spaces.
Appendix \ref{sa3}
discusses the energy of ground states on Euclidean space.
Such solutions are seen to have negative energy, and be energy minimizing,
when $1<p<1+4/n$, but not when $1+4/n<p<(n+2)/(n-2)$.
Appendix \ref{sa2} discusses cases when $F_\lambda$-minimizers can have
positive energy, even for $1<p<1+4/n$, in noneuclidean settings.

Appendix \ref{sa6} exhibits some positive solutions to
(\ref{E:SNLS}) that are not $F_\lambda$-minimizers, and cases where there are
two geometrically inequivalent, positive solutions to this equation.

\subsection{The concentration-vanishing-splitting trichotomy of Lions in a
  general setting}\label{sa1}

In this section we show the concentration-vanishing-splitting
trichotomy of Lions \cite{L1,L2} can be extended in a natural fashion
to a metric space setting.  

Let $X$ be a metric space and $\{ \mu_k \}$ a sequence of Borel
probability measures on $X$. 
For $R \in (0, \infty),\ y\in X$, set
$B_R(y) = \{ x \in X : d(x,y) \leq R \}$ to be the closed ball of
radius $R$, centered at $y$.  Define
\ben
Q_k(R) = \sup\limits_{y \in X}\,  \mu_k(B_R(y)).
\een
Of course each $Q_k$ is a monotone increasing function of $R$ on
$[0,\infty)$, and
\ben
\lim_{R \to \infty} Q_k(R) = 1.
\een
Using a standard
diagonalization procedure, we can reduce to a subsequence (which we
still denote by $\{ \mu_k \}$) such that $Q_k \to Q$ on $\mathbb{Q}^+$.
The function $Q$ is monotone increasing, so set
\ben
\alpha = \lim_{R \to \infty} Q(R) \in [0,1].
\een
We examine separately the three cases $\alpha = 0$, $\alpha = 1$, and
$0 < \alpha < 1$.  We will see that these three cases lead to the
phenomena of vanishing, concentration, or splitting respectively.
(Observe that Lions labels the third case ``dichotomy'' rather than
splitting.)

$\text{}$ \newline
{\bf Case I:}  $\alpha = 0$.  In this case,
\ben
\lim_{k \to \infty} \sup_{y \in X} \mu_k (B_R(y) ) = 0, \quad \forall
R < \infty.
\een
This is precisely the case of {\it vanishing}.

$\text{}$\newline
{\bf Case II:}  $\alpha = 1$.  In this case, for each $\mu \in (0,1)$,
there exists $R = R(\mu)$ such that, for every $k$, $Q_k(R(\mu)) >
\mu$.  That means there exist points $y_k(\mu) \in X$ such that
\ben
\mu_k(B_{R(\mu)}(y_k(\mu))) > \mu.
\een
Set $y_k = y_k(1/2)$, and observe that
\ben
\mu \geq 1/2 \implies d(y_k( \mu ), y_k ) \leq R(1/2) + R(\mu).
\een
This follows by definition, since otherwise there would be two
disjoint balls in $M$, each with $\mu_k$-measure exceeding $1/2$, which
contradicts $\mu_k$ being a probability measure.

As a consequence, with $y_k = y_k(1/2)$ as above, $\mu \in (1/2,1)$,
and 
\ben
\tR(\mu) = R(1/2) + 2 R(\mu),
\een
we have for all $k$
\ben
\mu_k( B_{\tR(\mu)}(y_k) ) > \mu.
\een
As this holds for each $\mu \in (1/2,1)$, this is the phenomenon of
{\it concentration.}

$\text{}$\newline
{\bf Case III:}  $0 < \alpha < 1$.  Pick $\epsilon>0$.  Then choose $R
\in (0,\infty)$ such that $Q(R) > \alpha - \epsilon.$  There exists
$k_0$ such that for each $k \geq k_0$, 
\begin{equation}
\label{E:Qk-limit}
\alpha - \epsilon < Q_k (R) < \alpha + \epsilon.
\end{equation}
We can also choose a sequence $R_k \to \infty$ such that
\ben
Q_k (R_k) \leq \alpha + \epsilon.
\een
By \eqref{E:Qk-limit}, there exist points $y_k \in X$ such that
\begin{equation}
\label{E:muk-BR}
\mu_k( B_R(y_k)) \in (\alpha -\epsilon, \alpha + \epsilon).
\end{equation}
Set
\ben
E_k^\sharp = B_R (y_k), \quad E_k^b = X \setminus B_{R_k} (y_k).
\een
Then
\ben
\dist (E_k^\sharp, E_k^b ) \geq R_k - R,
\een
and
\ben
\aligned
\mu_k(X) - \mu_k(E_k^\sharp) - \mu_k( E_k^b) & = \mu_k ( B_{R_k} (y_k)
\setminus B_R (y_k) )  \\
& \leq \alpha + \epsilon - (\alpha - \epsilon)  \\
& = 2 \epsilon. 
\endaligned
\label{E:muk-X}
\een
This is the phenomenon of {\it splitting}.  Observe that
\eqref{E:muk-BR} and \eqref{E:muk-X} imply
\ben
| \mu_k ( E^\sharp_k ) - \alpha | < \epsilon, \quad | \mu_k(E_k^b) -
(1 - \alpha ) | < 3 \epsilon.
\een

\subsection{Energy of ground states on Euclidean space}\label{sa3}

Assume $p\in (1,(n+2)/(n-2))$, and let $u_1>0$ satisfy
\ben
-\Delta u_1+u_1=u_1^p,\quad u_1\in H^1(\Rr^n).
\label{A.3.1}
\een
It follows from \cite{Kwo} and \cite{Mc} that such $u_1$ is unique, up to
a translation; it is radial and exponentially decreasing at infinity.
Such a solution is obtained as an $F_\lambda$-minimizer, with $\lambda=1$.
If we take $\lambda>0$ and set
\ben
u_\lambda(x)=\sigma^{2/(p-1)}u_1(\sigma x),\quad \lambda=\sigma^2,
\label{A.3.2}
\een
a calculation gives
\ben
-\Delta u_\lambda+\lambda u_\lambda=u_\lambda^p,
\label{A.3.3}
\een
Again, by the results cited above, $u_\lambda$ is the unique positive solution
in $H^1(\Rr^n)$ to such an equation, up to translation.  Calculations give

\ben
\aligned
\|u_\lambda\|^2_{L^2}
&=\lambda^{2/(p-1)-n/2} \|u_1\|^2_{L^2}, \\
\|\nabla u_\lambda\|^2_{L^2}
&=\lambda^{1+2/(p-1)-n/2} \|\nabla u_1\|^2_{L^2}, \\
\int|u_\lambda|^{p+1}\,dx
&=\lambda^{1+2/(p-1)-n/2} \int |u_1|^{p+1}\, dx,
\endaligned
\label{A.3.4}
\een
hence
\ben
E(u_\lambda)=\lambda^{1+2/(p-1)-n/2}E(u_1),
\label{A.3.5}
\een
while
\ben
Q(u_\lambda)=\lambda^{2/(p-1)-n/2}Q(u_1).
\label{A.3.6}
\een
Note that
\ben
\aligned
1+\frac{2}{p-1}-\frac{n}{2}>0&\Longleftrightarrow p<\frac{n+2}{n-2}, \\
\frac{2}{p-1}-\frac{n}{2}>0&\Longleftrightarrow p<1+\frac{4}{n}.
\endaligned
\label{A.3.7}
\een
Given $p>1$, the second restriction on $p$ in (\ref{A.3.7}) is equivalent to
(\ref{E:pp}).
If we set
\ben
e(\lambda)=E(u_\lambda),\quad q(\lambda)=\frac{1}{2}\|u_\lambda\|^2_{L^2},
\label{A.3.8}
\een
we get
\ben
e(\lambda)=\lambda^{1+\gamma}e(1),\quad q(\lambda)=\lambda^\gamma q(1),
\quad \gamma=\frac{2}{p-1}-\frac{n}{2}.
\label{A.3.9}
\een

The following general result can be combined with (\ref{A.3.9}) to
provide further information on $e(\lambda)$ and $q(\lambda)$.
(Further consequences are discussed in \cite{CMMT}.)

\begin{proposition}\label{pa.3.1}
Let $M$ be a complete Riemannian manifold and $u_\lambda$ a smooth family
of positive functions in $H^1(M)$, satisfying (\ref{A.3.3}).  Define
$q(\lambda)$ and $e(\lambda)$ by (\ref{A.3.8}).  Then
\ben
\frac{de}{d\lambda}=-\lambda \frac{dq}{d\lambda}.
\label{A.3.10}
\een
\end{proposition}
\begin{proof}
Take the inner product of (\ref{A.3.3}) with $\pa_\lambda u_\lambda$, to get
\ben
(\Delta u_\lambda-\lambda u_\lambda+u^p_\lambda,\pa_\lambda u_\lambda)=0.
\label{A.3.11}
\een
Note that
\ben
\frac{dq}{d\lambda}
=\frac{1}{2}\, \frac{\pa}{\pa\lambda}\|u_\lambda\|^2_{L^2}
=(u_\lambda,\pa_\lambda u_\lambda).
\label{A.3.12}
\een
Also,
\ben
\aligned
\frac{1}{2}\, \frac{\pa}{\pa\lambda}\|\nabla u_\lambda\|^2_{L^2}
&=(\nabla u_\lambda,\nabla\pa_\lambda u_\lambda) \\
&=-(\Delta u_\lambda,\pa_\lambda u_\lambda),
\endaligned
\label{A.3.13}
\een
and
\ben
\frac{1}{p+1}\frac{\pa}{\pa\lambda}\int_M u_\lambda^{p+1}\, dV
=\int_M u_\lambda^p\, \pa_\lambda u_\lambda\, dV,
\label{A.3.14}
\een
so
\ben
\frac{de_\lambda}{d\lambda}=
-(\Delta u_\lambda+u^p_\lambda,\pa_\lambda u_\lambda).
\label{A.3.15}
\een
Comparing (\ref{A.3.12}) and (\ref{A.3.15}) gives (\ref{A.3.10}),
via (\ref{A.3.11}).
\end{proof}

From this, we can deduce information about the sign of
$e(1)$, hence of $e(\lambda)$ for all $\lambda>0$.  By (\ref{A.3.9}) and
(\ref{A.3.7}),
\ben
\aligned
1<p<1+\frac{4}{n}&\Longrightarrow q'(\lambda)>0,\quad \forall\, \lambda>0 \\
&\Longrightarrow e'(\lambda)<0 \\
&\Longrightarrow e(1)<0 \\
&\Longrightarrow E(u_\lambda)<0,\quad \forall\, \lambda>0.
\endaligned
\label{A.3.16}
\een
On the other hand,
\ben
\aligned
1+\frac{4}{n}<p<\frac{n+2}{n-2}
&\Longrightarrow q'(\lambda)<0,\quad \forall\, \lambda>0 \\
&\Longrightarrow e'(\lambda)>0 \\
&\Longrightarrow e(1)>0 \\
&\Longrightarrow E(u_\lambda)>0,\quad \forall\, \lambda>0.
\endaligned
\label{A.3.17}
\een
(In the setting of (\ref{A.3.17}), (\ref{E:pp}) is violated,
and results of \S{\ref{s3}} do
not apply.)
From these observations, we can obtain positive and negative results about
energy minimizers.  Here is a positive result.

\begin{proposition}\label{pa.3.2}
If $p$ satisfies (\ref{A.3.16}) and $\lambda>0$, then the positive solution
$u\in H^1(\Rr^n)$ to (\ref{A.3.3}) is energy minimizing, within its mass
class.
\end{proposition}

\begin{proof}
Say $\|u_\lambda\|_{L^2}^2=\beta(\lambda)=2q(\lambda)$. Results of
\S{\ref{s3}} imply there exists a minimizer $v_\lambda$ for $E(v)$,
subject to the constraint $\|v\|_{L^2}^2=\beta(\lambda)$. We can assume
$v_\lambda>0$.  Furthermore, there exists $\mu\in\Rr$ such that
\ben
-\Delta v_\lambda+\mu v_\lambda=v_\lambda^p.
\label{A.3.18}
\een
We know that $\Spec(-\Delta+\mu)\subset[0,\infty)$, so, in this setting,
$\mu\geq 0$.  There are two possibilities: $\mu>0$ or $\mu=0$.
If $\mu>0$, the uniqueness result of \cite{Kwo}, \cite{Mc} implies $v_\lambda
=u_\mu$, up to a translation.  Now $\|v_\lambda\|_{L^2}^2=\|u_\mu\|_{L^2}^2
=\|u_\lambda\|_{L^2}^2$ implies $\mu=\lambda$, by (\ref{A.3.4}).  To finish,
we claim that $\mu=0$ is impossible.  Indeed, if this held, we could take the
inner product of (\ref{A.3.18}) with $v_\lambda$ to get
\ben
\|\nabla v_\lambda\|_{L^2}^2=\|v_\lambda\|_{L^{p+1}}^{p+1},
\label{A.3.19}
\een
hence
\ben
\aligned
E(v_\lambda)
&=\frac{1}{2}\|\nabla v_\lambda\|_{L^2}^2
-\frac{1}{p+1}\|v_\lambda\|_{L^{p+1}}^{p+1} \\
&=\Bigl(\frac{1}{2}-\frac{1}{p+1}\Bigr) \|v_\lambda\|_{L^{p+1}}^{p+1} \\
&>0.
\endaligned
\label{A.3.20}
\een
However, the minimum energy in this situation is $<0$, by (\ref{A.3.16}),
so the energy minimizer cannot satisfy (\ref{A.3.18}) with $\mu=0$.
\end{proof}

From this, we have the following existence and uniqueness result.

\begin{proposition}\label{pa.3.3}
If $p$ satisfies (\ref{A.3.16}) and $\beta>0$, there is a positive $u\in
H^1(\Rr^n)$ that minimizes $E(u)$ subject to the constraint $\|u\|_{L^2}^2
=\beta$.  This function solves (\ref{A.3.3}), for some $\lambda>0$, uniquely
determined by $\beta$, and it is unique up to translations.
\end{proposition}

\begin{proof} As seen in (\ref{A.3.20}), we cannot have $\lambda=0$.
It only remains to remark that (\ref{A.3.9}) sets up the unique correspondence
between $\beta$ and $\lambda$.
\end{proof}

We now record some negative results, when $p$ satisfies (\ref{A.3.17}).
Here is a preliminary result, using the fact (cf.~(\ref{E:i-beta})) that
\ben
\Cal{I}_\beta\rightarrow -\infty\ \text{ as }\ \beta\rightarrow+\infty.
\label{A.3.21}
\een

\begin{lemma} \label{la.3.4}
If $p$ satisfies (\ref{A.3.17}), then, as $\lambda\searrow 0$, $u_\lambda$
is not energy minimizing within its mass class.
\end{lemma}

\begin{proof} If (\ref{A.3.17}) holds, as $\lambda\searrow 0$, $\|u_\lambda
\|_{L^2}\nearrow +\infty$, so
\ben
\inf\, \{E(u):\|u\|_{L^2}=\|u_\lambda\|_{L^2}\}\rightarrow -\infty,
\label{A.3.27}
\een
by (\ref{A.3.21}).  Thus, by (\ref{A.3.17}), $E(u_\lambda)$ is not minimal.
\end{proof}

We can extend this result, as follows.
Given $\varphi_1\in H^1(\Rr^n)$, set
\ben
\varphi_\lambda(x)=\sigma^{2/(p-1)}\varphi_1(\sigma x),
\quad \lambda=\sigma^2,
\label{A.3.28}
\een
for $\lambda\in (0,\infty)$.  As in (\ref{A.3.4})--(\ref{A.3.6}),
\ben
\|\varphi_\lambda\|^2_{L^2}=\lambda^\gamma\|\varphi_1\|^2_{L^2},\quad
E(\varphi_\lambda)=\lambda^{1+\gamma}E(\varphi_1),
\label{A.3.29}
\een
with $\gamma$ given in (\ref{A.3.9}).  Given $\mu>0$ small enough, pick
$\varphi_1\in H^1(\Rr^n)$ such that (with $\varphi_\mu$ as in (\ref{A.3.29})
and $u_\mu$ as in (\ref{A.3.2}), where $\mu$ replaces $\lambda$)
\ben
\|\varphi_\mu\|_{L^2}=\|u_\mu\|_{L^2},\quad E(\varphi_\mu)<0.
\label{A.3.30}
\een
Then
\ben
\|\varphi_\lambda\|_{L^2}=\|u_\lambda\|_{L^2},
\quad E(\varphi_\lambda)<0,\quad
\forall\, \lambda\in (0,\infty).
\label{A.3.30A}
\een
This shows that, even for $p$ as in (\ref{A.3.17}),
\ben
\Cal{I}_\beta<0,\quad \forall\, \beta>0.
\label{A.3.31}
\een
Hence, for each $\lambda>0$, the positive solution $u_\lambda\in H^1(\Rr^n)$
to (\ref{A.3.3}) is not energy minimizing within its mass class.  Here is
a stronger result.

\begin{proposition} \label{pa.3.5}
If $p$ satisfies (\ref{A.3.17}), there is no energy minimizer $u\in
H^1(\Rr^n)$ within its mass class, at any positive mass.
\end{proposition}

\begin{proof}
Without loss of generality, such a minimizer can be taken to be $\geq 0$.
By (\ref{A.3.31}), such a minimizer must have negative energy.  It also must
solve (\ref{A.3.3}) for some $\lambda\geq 0$.  If $\lambda>0$, we contradict
the conclusion of (\ref{A.3.17}), and if $\lambda=0$ we contradict
(\ref{A.3.20}).
\end{proof}

\subsection{Ground states with positive energy}\label{sa2}

If $M$ is a weakly homogeneous space of dimension $n$, and
\ben
\Spec (-\Delta)\subset[\delta_0,\infty),\quad \lambda>-\delta_0,\quad
1<p<\frac{n+2}{n-2},
\label{A.1}
\een
then, as shown in \S{\ref{s2}},
one can minimize $F_\lambda(u)=\|\nabla u\|_{L^2}^2+\lambda
\|u\|_{L^2}^2$, subject to $\|u\|_{L^{p+1}}$ being fixed, and multiply by
a constant to get a positive solution to 
\ben
\Delta u_\lambda-\lambda u_\lambda+|u_\lambda|^{p-1}u_\lambda=0.
\label{A.2}
\een
As seen in \S{\ref{sa3}}, when $M=\Rr^n$ (where $\delta_0=0$),
these ground states all have negative energy when $1<p<1+4/n$,
and positive energy when $1+4/n<p<(n+2)/(n-2)$.
Here we note that some of these ``ground state'' solutions can have
positive energy, whenever (\ref{A.1}) holds with $\delta_0>0$,
even when $1<p<1+4/n$.

In fact, (\ref{A.2}) implies
\ben
\|\nabla u_\lambda\|^2_{L^2}=-\lambda\|u_\lambda\|^2_{L^2}
+\int\limits_M |u_\lambda|^{p+1}\, dV,
\label{A.3}
\een
which in turn implies
\ben
\aligned
E(u_\lambda)&=\frac{1}{2}\|\nabla u_\lambda\|^2_{L^2}
-\frac{1}{p+1}\int\limits_M |u_\lambda|^{p+1}\, dV \\
&=-\frac{\lambda}{2}\|u_\lambda\|^2_{L^2}
+\frac{p-1}{2(p+1)}\int\limits_M |u_\lambda|^{p+1}\, dV.
\endaligned
\label{A.4}
\een
If (\ref{A.1}) holds with $\delta_0>0$, we can pick $\lambda\in (-\delta_0,0]$
and find a ground state solution to (\ref{A.2})
(i.e., an $F_\lambda$-minimizer),
and then (\ref{A.4}) gives
\ben
\lambda\le 0\Longrightarrow E(u_\lambda)>0.
\label{A.5}
\een

For example, all the ground states on hyperbolic space $\Cal{H}^n$ associated
to $\lambda\le 0$ have positive energy.  It would be interesting to investigate
when they are energy minimizing, within their mass class.

\subsection{Non-$F_\lambda$-minimizers and related non-uniqueness}\label{sa6}

Fix $\lambda>0,\ \beta>0,\ n\ge 2,\ p\in (1,(n+2)/(n-2))$, and $R>0$, and set
\ben
\aligned
M&=\{x\in \Rr^n:|x|\geq R\},  \\
I_\beta&=\inf\,\{F_\lambda(u):u\in H^1_0(M),\, J_p(u)=\beta\}. 
\endaligned
\label{A6.2}
\een
As seen in \S{\ref{s2.a}}, there is no minimizer in such a case.  In fact,
$I_\beta$, given by (\ref{A6.2}), is equal to
$$
\inf\, \{F_\lambda(u):u\in H^1(\Rr^n),\, J_p(u)=\beta\},
$$
and $F_\lambda(u)>I_\beta$ for all $u\in H^1_0(M)$.  On the other hand,
methods of \S{\ref{sa4}} readily work to produce a minimizer for $F_\lambda$
restricted to the space $H^1_{0,r}(M)$ of {\it radial} functions in
$H^1_0(M)$, thus achieving
\ben
\aligned
&F_\lambda(v)=R_\beta,\qquad v\in H^1_{0,r}(M), \\
&R_\beta=\inf\, \{F_\lambda(u):u\in H^1_{0,r}(M),\, J_p(u)=\beta\}.
\endaligned
\label{A6.3}
\een
We can arrange that $v\ge 0$ on $M$.  Then $v$ is a radial solution to
\ben
-\Delta v+\lambda v=Kv^p,\quad K=\beta^{-1}R_\beta,\quad v\bigr|_{\pa M}=0,
\label{A6.4}
\een
and $v>0$ on the interior of $M$.  (Cf.~(\ref{1.1.4})--(\ref{E:K0})
for the computation of $K$.)  Of course, $v$ is not an $F_\lambda$-minimizer:
\ben
F_\lambda(v)>I_\beta.
\label{A6.5}
\een

We next construct some solutions on a compact, annular region.  With
$\lambda,\beta,n,p$, and $R$ as above, pick $S>R$ so large that one can take a
Euclidean $F_\lambda$-minimizer, translate it to be concentrated near a point
$p$, satisfying $|p|\sim (R+S)/2$, and cut it off near $|x|=R$ and $|x|=S$, in
such a way as to obtain
\ben
\tilde{u}\in H^1_0(\Omega),\quad J_p(\tilde{u})=\beta,\quad
F_\lambda(\tilde{u})<R_\beta,
\label{A6.6}
\een
where
\ben
\Omega=\{x\in\Rr^n:R\leq |x|\leq S\}.
\label{A6.7}
\een
Now, since $\Omega$ is compact, we can find minimizers for each of the
following:
\ben
\aligned
L_1&=\inf\, \{F_\lambda(u):u\in H^1_0(\Omega),\, J_p(u)=\beta)\}, \\
L_2&=\inf\, \{F_\lambda(u):u\in H^1_{0,r}(\Omega),\ J_p(u)=\beta\}.
\endaligned
\label{A6.8}
\een
Note that
\ben
L_1<R_\beta<L_2.
\label{A6.9}
\een
We can find
\ben
\aligned
w_1\in H^1_0(\Omega),\quad J_p(w_1)&=\beta,\quad F_\lambda(w_1)=L_1, \\
w_2\in H^1_{0,r}(\Omega),\quad J_p(w_2)&=\beta,\quad F_\lambda(w_2)=L_2,
\endaligned
\label{A6.10}
\een
and arrange that $w_j\ge 0$ on $\Omega$.
Then $w_j$ are positive solutions on $\Omega$ to
\ben
-\Delta w_j+\lambda w_j=K_j w_j^p,\quad K_j=\beta^{-1}L_j,\quad
w_j\bigr|_{\pa\Omega}=0.
\label{A6.11}
\een
Then
\ben
u_j=K_j^{1/(p-1)}w_j\in H^1_0(\Omega)
\label{A6.12}
\een
are positive solutions to
\ben
-\Delta u_j+\lambda u_j=u_j^p,\quad u_j\bigr|_{\pa\Omega}=0,
\label{A6.13}
\een
and, since (\ref{A6.9}) implies $K_1<K_2$, while $J_p(w_1)=J_p(w_2)$,
\ben
\|u_1\|_{L^{p+1}}<\|u_2\|_{L^{p+1}},
\label{A6.14}
\een
so these solutions are geometrically distinct.

$\text{}$ \newline
{\bf Remark.}  If $\lambda$ is large, the Euclidean $F_\lambda$-minimizer
is highly peaked, and $S$ need not be much larger than $R$.

\bibliographystyle{alpha}
\bibliography{MMT-bib}

\end{document}